\documentclass[12pt, twoside, leqno]{article}

\usepackage[hidelinks]{hyperref}

\usepackage[english]{babel}
\usepackage[T1]{fontenc}
\usepackage[utf8]{inputenc}
\usepackage[top=2.5cm,bottom=2cm,right=1.4cm,left=1.4cm]{geometry}
\usepackage{amsthm}
\usepackage{amsmath}
\usepackage{amsrefs}
\usepackage{amssymb}
\usepackage{mathtools}
\usepackage{sectsty}
\usepackage{mathrsfs}
\usepackage{hyperref}
\hypersetup{
    colorlinks=true,
    linkcolor=blue,
    filecolor=magenta,
    urlcolor=cyan,
    pdftitle={},
    pdfpagemode=FullScreen,
}
\usepackage{fancyhdr}
\usepackage{titlesec}
\titleformat{\section}{\normalfont\Large\bfseries}{\thesection}{1em}{\raggedright}

\newcommand{\R}{\mathbb{R}}
\newcommand{\C}{\mathbb{C}}
\newcommand{\N}{\mathbb{N}}
\newcommand{\Z}{\mathbb{Z}}

\renewcommand{\d}{\mathrm{d}}
\newcommand{\dt}{\, \mathrm{d}t}
\newcommand{\dx}{\, \mathrm{d}x}
\newcommand{\ds}{\, \mathrm{d}s}

\newcommand{\du}{\, \mathrm{d}u}

\newcommand{\dz}{\, \mathrm{d}z}

\newcommand{\sums}[1]{\underset{\hspace{-0.4em}#1}{\left.\sum \right.^{\ast}}}

\newcommand{\sumd}[1]{\underset{\hspace{-0.7em}#1}{\left.\sum \right.^{\#}}}

\newcommand{\eps}{\varepsilon}
\DeclarePairedDelimiter\abs{\lvert}{\rvert}
\renewcommand{\phi}{\varphi}
\renewcommand{\leq}{\leqslant}
\renewcommand{\geq}{\geqslant}

\DeclareMathOperator{\GL}{GL}
\DeclareMathOperator{\SL}{SL}
\DeclareMathOperator{\Res}{Res}
\newtheorem{thm}{Theorem}[section]

\newtheorem{lemme}{Lemma}[section]
\newtheorem{proposition}{Proposition}[section]
\theoremstyle{definition}

\newtheorem{rmq}{Remark}[section]

\newenvironment{preuve}
  {\noindent \textbf{Proof.}\\}
  {\begin{flushright} \qed\end{flushright}}

\numberwithin{equation}{section}

\sectionfont{\centering}

\title{Variance of $\GL(2)$ Fourier coefficients in arithmetic progressions}

\author{
Laurent Montaigu$^{1}$\\
\small
$^1$Univ. Bordeaux, CNRS, Bordeaux INP, IMB, UMR 5251,  F-33400, Talence, France
}
\date{}

\begin{document}
\maketitle
%\begin{center}
%		\vspace*{1cm}
%		\LARGE{\textbf{Variance of $\GL(2)$ Fourier coefficients in arithmetic progressions}}\\
%		\vspace*{0.5cm}
%		\normalsize
%		\textsc{Laurent Montaigu}
%\end{center}
\abstract{We improve a result of Lau and Zhao on the variance of Fourier 
coefficients of primitive cuspidal modular forms for $\SL_2(\Z)$ in 
arithmetic progressions. This is achieved by using bounds on the first 
moment of Rankin-Selberg $L$-functions in the height aspect and non-trivial estimates for shifted convolution sums.}

\renewcommand{\thefootnote}{}


\footnote{
Moments of $L$-functions. 
}

\renewcommand{\thefootnote}{\arabic{footnote}}
\setcounter{footnote}{0}

\tableofcontents
\section{Introduction}
\subsection{Variances of arithmetic functions}
The study of the fine distribution of primes in arithmetic progressions has a long history a cornerstone result of which is the Bombieri-Vinogradov Theorem
\begin{equation*}
	\sum_{q\leq \frac{N^{\frac{1}{2}}}{\log(N)^B}}\max_{x\leq N}\max_{(a,q)=1}\abs*{\sum_{\substack{n\leq x\\ \hspace{0.22cm} n\equiv a[q]\phantom{-}}}\Lambda(n)-\frac{x}{\phi(q)}}\ll N\log(N)^{-A},
\end{equation*}
where $\phi$ is Euler's totient function, $\Lambda$ is the von Mangoldt function, $A>0$ is arbitrary, and one can choose $B=3A+23$, see \cite{Vinogradov} and \cite{Bombieri}. Barban-Davenport-Halberstam showed in \cite{Barban} and \cite{DavenportHalberstam} a similar inequality for the mean square of the error term in the prime number Theorem in arithmetic progressions
\begin{equation}\label{BDH}
	\sum_{q\leq \frac{x}{\log(x)^B}} \sum_{\substack{a\leq q\\ \hspace{0.22cm} (a,q)=1\phantom{-}}}\abs*{\sum_{\substack{n\leq x\\ \hspace{0.22cm} n\equiv a[q]\phantom{-}}}\Lambda(n)-\frac{x}{\phi(q)}}^2\ll x^2\log(x)^{-A},
\end{equation}
with $B=A+5$, as an admissible value. Gallagher, in \cite{Gallagher}, then improved the error term to $Qx\log(x)$ when the sum over $q$ is replaced by $\sum_{q\leq Q}$ and $Q\gg x\log(x)^{-A}$ for some $A>0$, see \cite{Davenport} for further details. Note that the range on $q$ is much larger. Montgomery and Hooley showed that this result is essentially best possible as the main term on the left-hand side of \eqref{BDH} is $Qx\log(x)$, see \cite{Hooley} and  \cite{Montgomery}.
 
Over the years results of this type have had countless applications going sometimes beyond bounds that are conditional on RH (see for example \cite{Nguyen} and the references therein). Let us mention asymptotics for the variance and higher moments have been used to prove Gaussian equidistribution results; see, for example, \cite{FGKM}*{Theorem 1.1} and \cite{RicottaKowalski}*{Corollary D}. In \cite{BF}*{Corollary 1.4}, results on moments are used to obtain a lower bound for the error term in the prime number Theorem in arithmetic progressions.

Analogous variances for other arithmetically meaningful sequences have also been studied. For example, in \cite{Motohashi}, Motohashi replaced $\Lambda(n)$ by $\tau(n)=\sum_{d\mid n}1$ the divisor function, and found an asymptotic formula for 
\begin{equation} \label{divisorAP}
	\sum_{q\leq Q}A(x,q;\tau)\qquad A(x,q;\tau)=\sum_{b=1}^q\abs*{\sum_{\substack{n\leq x\\ \hspace{0.22cm} n\equiv b[q]\phantom{-}}}\tau(n)-\mathrm{MT}(b,x,q)}^2,
\end{equation}
where $Q=x$,
\begin{equation}\label{TPbxq}
	\mathrm{MT}(b,x,q)=\frac{1}{q}\sum_{r\mid (q,b)}\frac{\phi\left(\frac{q}{r}\right)}{\frac{q}{r}}x\left(\log(x)+2\gamma-1-2\log(r)\right)-\frac{2}{q}\sum_{r\mid (q,b)}\sum_{d\mid\frac{q}{r}}\frac{\mu(d)\log(d)}{d}x,
\end{equation}
the constant $\gamma$ is Euler-Mascheroni's constant and $(q,b)$ is the gcd of $q$ and $b$.
Lau and Zhao have studied in \cite{LauZhao} the following variance for $q\leq X$ 
\begin{equation}\label{fourierAP}
	A(X,q)=\sum_{b=1}^q\abs*{\sum_{\substack{n\leq X\\ \hspace{0.21cm} n\equiv b[q]\phantom{-}}}a(n)}^2,
\end{equation}
where $a(n)$ is the sequence of Hecke eigenvalue of a fixed modular form.

Several authors have worked with smoothed variants of \eqref{divisorAP} or \eqref{fourierAP}, see for example \cites{Nguyen,FGKM,RicottaKowalski,Lester}. Let us give some details on the results of \cite{FGKM}. The authors manage to prove the following asymptotic formula  
\begin{equation}\label{resultFGKM}
	\sum_{\substack{b=1\\ \hspace{0.22cm} (b,p)=1\phantom{-}}}^pS_f(X,p;b)^{2}=C_0X+O_{\eps}\left(p^{\frac{3}{2}+\eps}+\left(\frac{X^3}{p^2}\right)^{\frac{1}{2}+\eps}\right)
\end{equation}
for any $\eps>0$, where $C_0>0$ is explicit and $S_f(X,p,b)$ is defined by \eqref{a(n)APlisse}. 
Note that the asymptotic formula \eqref{resultFGKM} is non trivial in the range
\begin{equation*}
	X^{\frac{1}{2}+\eps}\leq p\leq X^{\frac{2}{3}-\eps},
\end{equation*}
for any $\eps>0$. This can be compared with the range obtained in \cite{LauZhao}*{Theorem 1}, given by 
\begin{equation*}
	X^{\frac{1}{2}}<p<X,
\end{equation*}
see also \cite{FGKM}*{Remark 1.3}. Thanks to the smooth sums, the authors manage to prove an asymptotic formula for every moment, and not only the second one. They are then able to prove the Gaussian equidistribution of the Hecke eigenvalue of a fixed modular form, see \cite{FGKM}*{Theorem 1.1}.
  
 The smooth sums Lau and Zhao dealt with are different from the ones in \cite{FGKM}, we emphasize the difference here. In \cite{FGKM} the authors essentially study sums of type  
\begin{equation}\label{a(n)APlisse}
	S_f(X,q,b)=\sum_{n\equiv b[q]}a(n)w\left(\frac{n}{X}\right),
\end{equation}
for some smooth function $w$ with compact support in $[1;2]$, $q\geq 1$ prime (we assume $q$ is prime here to simplify the argument but $q$ could in principle be any integer) and $b\in(\Z/q\Z)^*$. They detect the congruence condition $n\equiv b[q]$ with additive characters and then apply the Voronoi summation formula (that we recall in Proposition \ref{Voronoi}) and they get
\begin{equation*}
	\sum_{n=1}^{\infty}a(n)e\left(-\frac{\overline{b}n}{q}\right)\omega\left(\frac{n}{q^2/X}\right),
\end{equation*}
where, as usual, $e(x)=e^{2i\pi x}$  and $\omega$ is an integral transform of the function $w$ (see \cite{FGKM}*{Equation (3.2)}). The length of the dual sum is 
\begin{equation*}
	\frac{q^2}{X}X^{\eps}. 
\end{equation*}
This is to be compared with \cite{LauZhao} where the authors study the quantity $A(X,q)$ defined in \eqref{fourierAP}. They first smooth out the sum
\begin{equation*}
	A(X,q)\rightsquigarrow \sum_{b=1}^q\abs*{\sum_{n\equiv b[q]}a(n)w(n)}^2,
\end{equation*}
where $w$ is a smooth function supported on $[H;X]$ whose derivatives satisfy certain decay conditions and $q<H<X$ is a parameter optimized at the end of the proof (see \cite{LauZhao}*{Proposition 4.1} for their value of  $H$). Here again, we assume $q$ prime for simplicity. Just as in \cite{FGKM}, the congruence is detected using additive characters and then the authors apply Voronoi's summation formula. They are left with 
\begin{equation*}
	\sum_{n=1}^{\infty}a(n)e\left(-\frac{\overline{b}n}{q}\right)\omega_J\left(\frac{\sqrt{n}}{q}\right),
\end{equation*}
where $\omega_J$ is defined by \eqref{defomegaJ} (see \cite{LauZhao}*{Equation (4.5)}). In this case, the length of the dual sum is 
\begin{equation*}
	\frac{q^2}{X}\left(\frac{X}{H}\right)^2X^{\eps}\geq \frac{q^2}{X}X^{\eps}.
\end{equation*}
In both cases the length of the dual sum is shorter than the initial sum, but the saving obtained in \cite{FGKM} is much better.
\subsection{Statement of the main results}
Let $k\geq 4$ be an even integer and $f$ a weight $k$ cuspidal modular form for $\SL_2(\Z)$. We assume $f$ is an eigenvector of every Hecke operator with eigenvalue sequence $(a(n))_{n\geq 1}$. We also assume that  $f$ is primitive \textit{i.e.} $a(1)=1$. We can write the Fourier expansion of $f$ as 
\begin{equation*}
	\forall z\in\mathbb{H},\quad f(z)=\sum_{n=1}^{\infty}a(n)n^{\frac{k-1}{2}}e(nz), 
\end{equation*}
where as usual $\mathbb{H}=\{z\in\C\mid \Im(z)>0\}$ is the upper half plane.\bigbreak
\noindent
The two main objects of study in this article are the following.
	Let $X\geq 1$ be a real number and $q\leq X$ be an integer. Recall the definition of the variance of Fourier coefficients in arithmetic progressions $A(X,q)$ defined by \eqref{fourierAP} and the variance of the divisor function in arithmetic progressions $A(X,q;\tau)$ defined by \eqref{divisorAP}.
	 
Several authors studied $A(X,q)$ or close variants. Blomer proved in \cite{Blomer} that for any $\eps>0$
\begin{equation*}
	A(X,q)\ll_{\eps} X^{1+\eps}. 
\end{equation*}
Guangshi sharpened Blomer's result in \cite{Guangshi} and showed that 
\begin{equation*}
	A(X,q)\ll X(\log(X))^2,
\end{equation*}
using Wilton's estimate $\sum_{n\leq x}a(n)e(\beta n)\ll x^{\frac{1}{2}}\log(x)$. Finally, Lau and Zhao greatly improved the last two estimates in \cite{LauZhao} and found a transition behavior for $A(X,q)$ and $A(X,q;\tau)$ at $q\approx X^{\frac{1}{2}}$. More precisely, they showed for $B(X,q)$ either equal to $A(X,q)$ or to $A(X,q;\tau)$ that
\begin{itemize}
	\item For $1\leq q\ll_{\eps} X^{\frac{1}{4}+\eps}$,
\begin{equation*}
	B(X,q)\ll q^{\frac{1}{3}}X^{\frac{2}{3}+\eps}.
\end{equation*}
	\item For $X^{\frac{1}{4}+\eps}\ll q\ll X^{\frac{1}{2}-\eps}$,
\begin{equation*}
	B(X,q)\asymp q^{-1}\sum_{r\mid q}r\phi(r)X^{\frac{1}{2}}.
\end{equation*}
	\item For $X^{\frac{1}{2}}<q<X$,
\begin{equation*}
	A(X,q)=cX+O\left(q^{\frac{1}{6}}X^{\frac{5}{6}}\tau(q)+q^{-\frac{1}{2}}X^{\frac{5}{4}}g(q)\right),
\end{equation*}
where $c$ is the residue of the Rankin-Selberg $L$-function $L(f\otimes f,s)$ at $s=1$ and $g(q)=\sum_{r\mid q}\phi(r)r^{-1}$. 
	\item For $X^{\frac{1}{2}}<q<X$,
\begin{equation*}
	A(X,q;\tau)=\frac{X}{q}\sum_{r\mid q}\phi(r)P_3\left(\log\left(\frac{r^2}{X}\right)\right)+O\left(q^{\frac{1}{6}}X^{\frac{5}{6}}\tau(q)\log(X)^3+q^{-\frac{1}{2}}X^{\frac{5}{4}}g(q)\right), 
\end{equation*}
where $P_3$ is a polynomial of degree $3$ with positive leading coefficient. 
\end{itemize}

The goal of this article is to improve Lau and Zhao's results in the range $X^{\frac{1}{2}}<q<X$. Precisely, we will prove the following result in the cuspidal case.
\begin{thm}\label{moduleqqc}
	Assume $X^{\frac{1}{2}}<q<X$, then 
\begin{equation*}
	A(X,q)=cX+O\left(q^{-1}X^{\frac{3}{2}}g(q)+q^{\frac{5}{54}}X^{\frac{47}{54}}X^{\eps}+q^{\frac{1}{2}}X^{\frac{1}{2}}\right).
\end{equation*} 
\end{thm}

\begin{rmq}
Comparing with \cite{LauZhao}*{Theorem 1}, we do obtain an improved error term since for $q>X^{\frac{1}{2}}$
\begin{equation*}
	q^{-1}X^{\frac{3}{2}}g(q)\ll q^{-\frac{1}{2}}X^{\frac{5}{4}}g(q),
\end{equation*}
\begin{equation*}
	q^{\frac{5}{54}}X^{\frac{47}{54}}X^{\eps}+q^{\frac{1}{2}}X^{\frac{1}{2}}\ll q^{\frac{1}{6}}X^{\frac{5}{6}}\tau(q).
\end{equation*}
\end{rmq}

\begin{rmq}
In Theorem \ref{moduleqqc}, the main error term is  $O(q^{-1}X^{\frac{3}{2}}g(q))$ for $X^{\frac{1}{2}}\leq q\leq X^{\frac{34}{59}}$, the main error term is  $O\left(q^{\frac{5}{54}}X^{\frac{47}{54}}X^{\eps}\right)$ for $X^{\frac{34}{59}}\leq q\leq X^{\frac{10}{11}}$.  
\end{rmq}

\begin{rmq}
In Theorem \ref{moduleqqc}, the error term $O\left(X^{\frac{1}{2}}q^{\frac{1}{2}}\right)$ arises from the Cauchy--Schwarz inequality applied at the outset of the proof and can be regarded as the best possible error term obtainable by this method, see also Remark \ref{q1/2}. 
\end{rmq}

For the case of the divisor function, our improvement is the following.
\begin{thm}\label{thdivisor}
	Assume $X^{\frac{1}{2}}<q<X$, then
\begin{align*}
	&A(X,q;\tau)=\frac{X}{q}\sum_{r\mid q}\phi(r)P_3\left(\frac{r^2}{X}\right)\\
	&+O\left(q^{-1}X^{\frac{3}{2}}g(q)+q^{\frac{1}{2}}X^{\frac{1}{2}}\tau(q)\log(X)^{4+\eps} +q^{\frac{1}{4}}X^{\frac{3}{4}}(\log(X)^3+\log(X)^{\frac{5}{2}+\eps}\tau(q)) +q^{\frac{5}{36}}X^{\frac{61}{72}}X^{\eps}\right),
\end{align*}
where $P_3(\xi)$ is a polynomial in $\log(\xi)$ of degree $3$ with positive leading coefficient (the same as in \cite{LauZhao}*{Theorem 3}). 
\end{thm}

\begin{rmq}
Comparing with \cite{LauZhao}*{Theorem 2}, we do obtain an improved error term since for $X^{\frac{1}{2}}<q<X^{1-\eps}$ for any $\eps>0$
\begin{equation*}
	q^{-1}X^{\frac{3}{2}}g(q)\ll q^{-\frac{1}{2}}X^{\frac{5}{4}}g(q),
\end{equation*}
\begin{equation*}
	q^{\frac{1}{2}}X^{\frac{1}{2}}\tau(q)\log(X)^{4+\eps} +q^{\frac{1}{4}}X^{\frac{3}{4}}(\log(X)^3+\log(X)^{\frac{5}{2}+\eps}\tau(q)) +q^{\frac{5}{36}}X^{\frac{61}{72}}X^{\eps}\ll q^{\frac{1}{6}}X^{\frac{5}{6}}\tau(q)\log(X)^3. 
\end{equation*}
\end{rmq}

\begin{rmq}
In Theorem \ref{thdivisor}, the main error term is  $O\left(q^{-1}X^{\frac{3}{2}}g(q)\right)$ 
 for $X^{\frac{1}{2}}\leq q\leq X^{\frac{47}{82}}$, the main error term is  $O\left(q^{\frac{5}{36}}X^{\frac{61}{72}}X^{\eps}\right)$ for $X^{\frac{47}{82}}\leq q\leq X^{\frac{7}{8}}$.  
\end{rmq}
Let us mention we can sligthly improve the result in both Theorem \ref{moduleqqc} and Theorem \ref{thdivisor} if one assumes a wider zero-free region for the Riemann $\zeta$-function than the one currently known. Precisely, let $\delta\in[0;\frac{1}{12}]$, we state the following hypothesis on the zeros of the Riemann $\zeta$ function.
\begin{equation}\label{hypo}\tag{$H_{\delta}$}
	\zeta \text{ has no zeros on the half plane }\{\sigma\geq \frac{2}{3}+4\delta\}.
\end{equation}
Currently, the only value of $\delta\in[0;\frac{1}{12}]$ for which we know \eqref{hypo} is true is $\delta=\frac{1}{12}$.
Under \eqref{hypo}, we can replace the error term $q^{-1}X^{\frac{3}{2}}g(q)$ in both Theorem \ref{moduleqqc} and Theorem \ref{thdivisor} by 
\begin{equation}\label{MET}
	E_{\delta}=\min\left(\widetilde{E_{\delta}},q^{-1}X^{\frac{3}{2}}g(q)\right),
\end{equation}
where
\begin{equation*}
	\widetilde{E_{\delta}}:= \frac{X^{\frac{5}{3}}}{q^{\frac{4}{3}}}\tau(q)+\frac{X^{\frac{5}{3}}}{qd_{\frac{1}{2}}(q)^{\frac{1}{3}}} \left(\frac{q^2}{X}\right)^{2\delta}\sum_{\substack{r\mid q\\ \hspace{0.22cm}X^{\frac{1}{2}}\leq r\phantom{-}}}\frac{\phi(r)}{r} +\frac{X^{\frac{3}{2}}}{q}\sum_{\substack{r\mid q\\ \hspace{0.22cm} X^{\frac{2}{3}}q^{-\frac{1}{3}} <r<X^{\frac{1}{2}}\phantom{-}}}\frac{\phi(r)}{r}, 
\end{equation*}
where $d_{\frac{1}{2}}(q)$ is the smallest divisor of $q$ which is $> X^{\frac{1}{2}}$. 

 	Let us examine the error term $\widetilde{E_{\delta}}$. First, it is easy to see that $\widetilde{E_{\delta}}\ll q^{-\frac{2}{3}}X^{\frac{4}{3}}g(q)$. In some cases we can do better. For example, if $q$ is prime then the sum
\begin{equation*}
 	\sum_{\substack{r\mid q\\ \hspace{0.22cm} X^{\frac{2}{3}}q^{-\frac{1}{3}} <r<X^{\frac{1}{2}}\phantom{-}}}\frac{\phi(r)}{r},
\end{equation*}
vanishes as the index set is empty. We also have $d_{\frac{1}{2}}(q)=q$ because $q>X^{\frac{1}{2}}$, hence in the case $q$ prime and under \eqref{hypo} we have
\begin{equation*}
	A(X,q)=cX+O\left(\frac{X^{\frac{5}{3}}}{q^{\frac{4}{3}}}\left(\tau(q)+\left(\frac{q^2}{X}\right)^{2\delta}\right)+q^{\frac{5}{54}}X^{\frac{47}{54}}X^{\eps}+q^{\frac{1}{2}}X^{\frac{1}{2}}\right).
\end{equation*}
We can generalize this. Let $P^{-}(q)$ denote the smallest prime factor of $q$. Then, the biggest divisor of $q$ that is not $q$ is 
\begin{equation*}
	\frac{q}{P^{-}(q)}.
\end{equation*}
If we assume this divisor is $<X^{\frac{2}{3}}q^{-\frac{1}{3}}$ that is
\begin{equation*}
	P^{-}(q)>q^{\frac{4}{3}}X^{-\frac{2}{3}}, 
\end{equation*}
then the two conditions $r\mid q$ and $X^{\frac{2}{3}}q^{-\frac{1}{3}}<r<X^{\frac{1}{2}}$ are incompatible and we get
\begin{equation*}
	\widetilde{E_{\delta}}\ll \frac{X^{\frac{5}{3}}}{q^{\frac{4}{3}}}\left(\frac{q^2}{X}\right)^{2\delta}. 
\end{equation*}
Finally, we have $E\ll q^{\frac{5}{54}}X^{\frac{47}{54}}$ for $q\geq X^{\frac{34}{59}}$, hence we can check that for $k\leq 5$ and
\begin{equation*}
	q=p^k\leq X^{\frac{34}{59}}
\end{equation*}
we have
\begin{equation*}
	\widetilde{E_{\delta}}\ll \frac{X^{\frac{5}{3}}}{q^{\frac{4}{3}}}\left(\tau(q)+\left(\frac{q^2}{X}\right)^{2\delta}\right).
\end{equation*}

\subsection{Sketch of the proof}

Our strategy of proof follows the approach of \cite{LauZhao}. There are two main differences between our method and the one used by Lau-Zhao. We give some details for the case of $A(X,q)$. \smallbreak
\noindent
First, in \cite{LauZhao}, the main term $cX$ is extracted using an Abel summation and the Rankin-Selberg estimate
\begin{equation*}
	\sum_{n\leq x}a(n)^2=cx+O(x^{\frac{3}{5}}).
\end{equation*}
We now know (see \cite{Huang}) that the exponent $\frac{3}{5}$ is not optimal but it does not help us because the Abel summation involves the derivative of the function $\omega$ appearing when one applies Voronoi's summation formula. Due to the choice of function $w$, the $\omega$ function has large derivatives. In this article, we extract the main term using the inverse Mellin transform on the $\omega$ function. Doing so, we are left with an integral of an integrand of the form
\begin{equation*}
	L(f\otimes f,s)\times\text{ quotient of $\Gamma$ functions},
\end{equation*}
where $L(f\otimes f,s)$ is a Rankin-Selberg $L$-function.  
We can bound this integral using Stirling's formula and a bound on the first moment of $L(f\otimes f,\sigma+it)$ for $\sigma\in ]0;1[$.\smallbreak
\noindent
Our second main contribution relies on a result in \cite{quadraticdivisorproblem} about shifted convolution sums. We will use Proposition \ref{problemconv} to bound the shifted convolution sum arising from Voronoi's summation formula. Such sums are roughly of the form 
\begin{equation*}
	\sum_{n\equiv m[q]}a(n)a(m)\omega\left(\frac{\sqrt{m}}{q}\right)\omega\left(\frac{\sqrt{n}}{q}\right).
\end{equation*}
In \cite{LauZhao}, this sum is bounded almost trivially by combining Abel summation with Rankin's estimate. Using Proposition \ref{problemconv}, we can exploit the oscillations of the Fourier coefficients of $f$. This leads to an extra saving in the range $q\geq X^{\frac{1}{2}}$.\smallbreak
\noindent
The arguments are similar for the proof of Theorem \ref{thdivisor} but the details are more involved.\bigbreak
\noindent
The structure of the article is the following. In section \ref{cplxana} we recall some useful standard results in complex analysis and modular forms. In section \ref{specialfunction} we prove preparatory  Lemmas on the function which appears in the Voronoi summation formula. The two main sections of the paper are Sections \ref{proofTh} and \ref{divisor}, where we respectively prove Theorem \ref{moduleqqc} and Theorem \ref{thdivisor}.
\bigbreak
\noindent 
\textbf{Acknowledgement} : I thank Yuk-Kam Lau for pointing out a mistake in a preliminary version of this work. I also thank Etienne Fouvry and Emmanuel Kowalski for their helpful comments in the preparation of this work.   
\bigbreak
\noindent 
\textbf{Notations} : An integral $\int_{(c)}$ means $\int_{c-i\infty}^{c+i\infty}$. A sum $\sum_{d\,(q)}$ is a sum over $d\in\Z/q\Z$ and a sum $\sums{d\,(q)}$ is a sum over $d\in(\Z/q\Z)^*$. The sum $c_k(h)=\sums{d\,(k)}e\left(\frac{hd}{k}\right)$ is the Ramanujan sum. The sum $\sumd{M}$ means the index set of the sum is dyadically split. By $n\sim N$ we mean $N\leq n\leq 2N$ and by $A\asymp B$ we mean $A\ll B\ll A$. The quantity $\eps>0$ is arbitrarily small and may not be the same at each occurence. 
\section{Preliminaries and background}\label{cplxana}
\subsection{Standard analytic and modular tools}
We begin this subsection by recalling some well-known formulas from complex analysis.
\begin{lemme}\label{Stirling}
	Let $\Gamma$ be the Euler $\Gamma$ function defined by $\Gamma(z)=\int_0^{\infty}t^{z-1}e^{-t}\dt$. We have the following asymptotic formulas
\begin{enumerate}
	\item Let $z\in\C$ with $\abs{\arg(z)}\leq \pi-\eps$
\begin{equation*}
	\Gamma(z)=\left(\frac{2\pi}{z}\right)^{\frac{1}{2}}\left(\frac{z}{e}\right)^z\left(1+O\left(\frac{1}{\abs{z}}\right)\right).
\end{equation*}
	\item For $\sigma\in\R$ fixed and $t\neq 0$
\begin{equation*}
	\Gamma(\sigma+it)=\sqrt{2\pi}\abs{t}^{\sigma-\frac{1}{2}}e^{-\frac{\pi}{2}\abs{t}}\left(\frac{\abs{t}}{e}\right)^{it}e^{i\frac{\pi}{2}(\sigma-\frac{1}{2})}\left(1+O\left(\frac{1}{\abs{t}}\right)\right).
\end{equation*}
\item For $\sigma\in\R$ fixed and $t\neq 0$
\begin{equation*}
	\abs{\Gamma(\sigma+it)}=\sqrt{2\pi}\abs{t}^{\sigma-\frac{1}{2}}e^{-\frac{\pi}{2}\abs{t}} \left(1+O\left(\frac{1}{\abs{t}}\right)\right). 
\end{equation*}
\item For $\sigma\in\R$ fixed and $\abs{t}\geq 2$
\begin{equation*}
	\abs{\Gamma'(\sigma+it)}\asymp \abs{\log(t)}\abs{\Gamma(\sigma+it)}.
\end{equation*}
\end{enumerate}
\end{lemme}

\begin{preuve}
1. This result can be found in \cite{Iwakow}*{Equation (5.112)}.\smallbreak
\noindent 2., 3. follow easily from $1.$\smallbreak
\noindent 
4. If $\sigma>0$, this follows from $3$ and the relation (see \cite{Whittaker}*{Example~12.3.7})
\begin{equation*}
	\forall z\in\C,\,\Re(z)>0,\quad \frac{\Gamma'(z)}{\Gamma(z)}=\log(z)-2\gamma-2\int_0^{\infty}\frac{\dt}{(t^2+z^2)(e^{2\pi t}-1)}.
\end{equation*}
If $\sigma\geq -1$, we use the relation 
\begin{equation*}
	\frac{\Gamma'(z)}{\Gamma(z)}=\frac{\Gamma'(z+1)}{\Gamma(z+1)}-\frac{1}{z}. 
\end{equation*}
We extend the result to any $\sigma\in\R$ by induction. 
\end{preuve}

Using Stirling's formula we can prove the following Lemma.
\begin{lemme}\label{estimF}
	Let $F$ be defined by $F(s)=2\pi i^k(2\pi)^{-2s}\frac{\Gamma\left(\frac{k-1}{2}+s\right)}{\Gamma\left(\frac{k+1}{2}-s\right)}$, then for $\sigma>-\frac{k-1}{2}$ fixed
\begin{equation*}
	\abs{F(\sigma+it)}\asymp \frac{1}{(1+\abs{t})^{1-2\sigma}}.
\end{equation*}
\end{lemme}
\begin{preuve}
The result is clear if $\abs{t}\leq 1$. If $\abs{t}\geq 1$ we use Stirling's formula (Lemma \ref{Stirling}) to get
\begin{equation*}
	\abs*{F(\sigma+it)}\asymp \abs*{\frac{\Gamma\left(\frac{k-1}{2}+\sigma+it\right)}{\Gamma\left(\frac{k+1}{2}-\sigma-it\right)}}\asymp \frac{\abs{t}^{\frac{k-1}{2}+\sigma-\frac{1}{2}}e^{-\frac{\pi}{2}\abs{t}}}{\abs{t}^{\frac{k+1}{2}-\sigma-\frac{1}{2}}e^{-\frac{\pi}{2}\abs{t}}}\asymp \frac{1}{(1+\abs{t})^{1-2\sigma}}. 
\end{equation*}	
\end{preuve}

The following lemma is a Mellin-Barnes integral representation.
\begin{lemme}
Let $A$ and $B$ be positive real numbers. For $\lambda\in\C$
\begin{equation}\label{(A+B)lambda}
	(A+B)^{-\lambda}=\frac{1}{\Gamma(\lambda)2i\pi}\int_{(c)}\Gamma(\lambda+z)\Gamma(-z)A^{-\lambda-z}B^z\dz,
\end{equation}
\begin{equation}\label{log(A+B)(A+B)lambda}
\begin{split}
\log(A+B)(A+B)^{-\lambda} = & \frac{\Gamma'(\lambda)}{2i\pi\Gamma(\lambda)^2}\int_{(c)}\Gamma(\lambda+z)\Gamma(-z)A^{-\lambda-z}B^{z}\dz \\
 &-\frac{1}{2i\pi\Gamma(\lambda)}\int_{(c)}\Gamma'(\lambda+z)\Gamma(-z)A^{-\lambda-z}B^{z}\dz\\
 &+\frac{1}{2i\pi\Gamma(\lambda)}\int_{(c)}\Gamma(\lambda+z)\Gamma(-z)\log(A)A^{-\lambda-z}B^{z}\dz,
\end{split}
\end{equation}
where, in both equations, the path of integration separates the poles of $\Gamma(\lambda+z)$ from those of $\Gamma(-z)$.
\end{lemme}

\begin{preuve}
Equation \eqref{(A+B)lambda} can be proved using \cite{Whittaker}*{Corollary~14.5.1}. Equation \eqref{log(A+B)(A+B)lambda} is obtained by differentiating \eqref{(A+B)lambda} with respect to $\lambda$.  	
\end{preuve}

\begin{lemme}
	Define for $n\geq 1$, $a\in\C$ and $k\in\{0,1,2\}$
\begin{equation*}
	\sigma_a^{(k)}(n)=\sum_{d\mid n}d^{a}\log(d)^k.
\end{equation*}
Then for $r\in\N^*$, $a\neq 0$ and $s\in\C$ with $\Re(s)>\max(1,\Re(a)+1)$ 
\begin{equation}\label{sumsigma_a}
	\sum_{n=1}^{\infty}\frac{\sigma_{a}^{(k)}(nr)}{n^s}=\frac{\partial^k}{\partial a^k}\left(\zeta(s)\zeta(s-a)h(a,r,s)\right),
\end{equation}
where 
\begin{equation*}
	h(a,r,s)= \prod_{p\mid r}(1-p^{a})^{-1}\left(1-p^{-(s-a)}-p^{a(v_p(r)+1)}+p^{-s+a(v_p(r)+1)}\right).
\end{equation*}
The function $s\mapsto h(a,r,s)$ is entire and satisfies for $\Re(s)>0$, $k\in\{0,1,2\}$ 
\begin{equation}\label{h(a,r,s)}
	\left.\frac{\partial^k}{\partial a^k}h(a,r,s)\right\rvert_{a=-1}\ll \prod_{p\mid r}\left(1+\frac{1}{p-1}\right)=:h(r)
\end{equation}
 
\end{lemme}
\begin{preuve}
We start with the case $k=0$ and write $\sigma_a$ for $\sigma_a^{(0)}$. It is well known that
\begin{equation*}
	\sum_{n=1}^{\infty}\frac{\sigma_a(n)}{n^s}=\zeta(s)\zeta(s-a)=\prod_{p}(1-p^{-s})^{-1}(1-p^{-(s-a)})^{-1}.
\end{equation*}
Let $r=\prod_{p\mid r}p^{v_p(r)}$. We have
\begin{align*}
	\sum_{n=1}^{\infty}\frac{\sigma_{a}(nr)}{(nr)^s}&=\prod_{(p,r)=1}(1-p^{-s})^{-1}(1-p^{-(s-a)})^{-1}\prod_{p\mid r}\left(\sum_{k=v_p(r)}^{\infty}\frac{\sigma_a(p^{k})}{p^{ks}}\right)\\
	&=\prod_{(p,r)=1}(1-p^{-s})^{-1}(1-p^{-(s-a)})^{-1}\prod_{p\mid r}p^{-v_p(r)s}\sum_{k=0}^{\infty}\frac{\sigma_a(p^{k+v_p(r)})}{p^{ks}}\\
	&=r^{-s}\zeta(s)\zeta(s-a)\prod_{p\mid r}(1-p^{-s})(1-p^{-(s-a)})\sum_{k=0}^{\infty}\frac{\sigma_a(p^{k+v_p(r)})}{p^{ks}}.
\end{align*}
We can compute the $k$-sum using geometric series 
\begin{equation*}
	\sum_{k=0}^{\infty}\frac{\sigma_a(p^{k+v_p(r)})}{p^{ks}}=(1-p^{a})^{-1}\left(\frac{1}{1-p^{-s}}-\frac{p^{a(v_p(r)+1)}}{1-p^{-(s-a)}}\right).
\end{equation*}
Hence,
\begin{align*}
	\sum_{n=1}^{\infty}\frac{\sigma_{a}(nr)}{n^s}&=\zeta(s)\zeta(s-a)\prod_{p\mid r}(1-p^{a})^{-1}\left(1-p^{-(s-a)}-p^{a(v_p(r)+1)}+p^{-s+a(v_p(r)+1)}\right)\\
	&=\zeta(s)\zeta(s-a)h(a,r,s),
\end{align*}
which proves \eqref{sumsigma_a} for $k=0$. To prove \eqref{sumsigma_a} for $k\in\{1,2\}$ we just need to see that 
\begin{equation*}
	\sigma_a^{(k)}(n)=\frac{\partial^k}{\partial a^k}\sigma_a(n).
\end{equation*}
We now check that 
\begin{equation*}
	h(-1,r,s)\ll \prod_{p\mid r}\frac{1}{1-p^{-1}}=\prod_{p\mid r}\left(1+\frac{1}{p-1}\right),
\end{equation*}
\begin{equation*}
	\left.\frac{\partial}{\partial a}h(a,r,s)\right\rvert_{a=-1}\ll \prod_{p\mid r}\frac{\log(p)p^{-1}}{(1-p^{-1})^2}+ \prod_{p\mid r}\left(1+\frac{1}{p-1}\right)\ll \prod_{p\mid r}\left(1+\frac{1}{p-1}\right)
\end{equation*}
and 
\begin{equation*}
	\left.\frac{\partial^2}{\partial a^2}h(a,r,s)\right\rvert_{a=-1}\ll \prod_{p\mid r}\frac{\log(p)^2p^{-1}}{(1-p^{-1})^2}+\prod_{p\mid r}\frac{2\log(p)p^{-2}}{(1-p^{-1})^3}+\prod_{p\mid r}\left(1+\frac{1}{p-1}\right)\ll \prod_{p\mid r}\left(1+\frac{1}{p-1}\right),
\end{equation*}
which proves \eqref{h(a,r,s)}. 
\end{preuve}
In the next Lemma we bound $h(n)$ defined in \eqref{h(a,r,s)}.
\begin{lemme}\label{borneh(n))}
We have for $n\geq 10$
\begin{equation*}
h(n)\ll \log(\log(n)). 
\end{equation*}
\end{lemme}

\begin{preuve}
It is clear that $h(n)$ is maximal for primorial $n$ \textit{i.e.}
\begin{equation*}
	n=\prod_{p\leq x}p
\end{equation*}
for some real number $x$. We have 
\begin{equation*}
	h(n)=\prod_{p\leq x}\left(1+\frac{1}{p-1}\right)\ll \prod_{3\leq p\leq x}\frac{1-\frac{1}{(p-1)^2}}{1-\frac{1}{p-1}}\ll \log(x),
\end{equation*}
where we used Merten's third theorem. 
By the prime number theorem we have
\begin{equation*}
	\log(n)=\sum_{p\leq x}\log(p)\sim x,
\end{equation*}
hence $\log(x)\ll \log(\log(n))$ and finally
\begin{equation*}
	h(n)\ll \log(\log(n)). 
\end{equation*}
\end{preuve}

\begin{lemme}\label{Diricraman}
	We have for $\Re(s)>1$ and $h\in\Z$ 
\begin{equation*}
	\sum_{k=1}^{\infty}\frac{c_k(h)}{k^s}=\frac{\sigma_{1-s}(h)}{\zeta(s)}.
\end{equation*}
\end{lemme}
\begin{preuve}
This follows from the formula $c_k(n)=\sum_{d\mid k}\mu\left(\frac{k}{d}\right)u_d$ where $u_d=d$ if $d\mid n$ and $0$ otherwise. 	
\end{preuve}

The next lemma concerns the Mellin transform of Bessel functions. Background on Bessel functions can be found in \cite{Treatise}. We recall their definition here.  
The function $J_{\nu}$ is the Bessel function of the first kind defined by
\begin{equation*}
	\forall \nu\in\C,\forall x>0,\quad J_{\nu}(x)=\sum_{m=0}^{\infty}\frac{(-1)^m}{m!\Gamma(m+\nu+1)}\left(\frac{x}{2}\right)^{2m+\nu}.
\end{equation*}
The function $Y_0$ is the Bessel function of the second kind defined by
\begin{equation}\label{defY0}
	\forall x>0,\quad Y_0(x)= \lim_{\substack{\nu\to 0\\ \hspace{0.22cm} \nu\neq 0\phantom{-}}}\frac{J_{\nu}(x)\cos(\pi\nu)-J_{-\nu}(x)}{\sin(\pi\nu)}.
\end{equation}
The function $K_0$ is the modified Bessel function of the second kind defined by
\begin{equation}\label{defK0}
	\forall x>0,\quad K_{0}(x)=\lim_{\substack{\nu\to 0\\ \hspace{0.22cm} \nu\neq 0\phantom{-}}}\frac{\pi}{2}\left(\frac{I_{-\nu}(x)-I_{\nu}(x)}{\sin(\pi\nu)}\right),
\end{equation}
where for $x>0$, $\nu\in\C$
\begin{equation*}
	I_{\nu}(x)=\sum_{m=0}^{\infty}\frac{1}{m!\Gamma(\nu+m+1)}\left(\frac{x}{2}\right)^{2m+\nu}.
\end{equation*}
\begin{lemme}\label{mellinbessel}
For $x>0$ and $k\geq 1$ 
\begin{equation*}
	\left(\frac{x}{2}\right)^{-(k-1)}J_{k-1}(x)=\frac{1}{2i\pi}\int_{(\sigma)}\frac{\Gamma(s)}{\Gamma(k-s)}\left(\frac{x}{2}\right)^{-2s}\ds,\quad 0<\sigma<\frac{k}{2}. 
\end{equation*}
For $s\in\C$, $\Re(s)\in ]0;\frac{3}{2}[$
\begin{equation*}
	\int_0^{\infty}Y_0(x)x^{s-1}\dx=-\frac{2^{s-1}}{\pi}\cos\left(\frac{\pi s}{2}\right)\Gamma\left(\frac{s}{2}\right)^2,
\end{equation*}
the integral being absolutely convergent for $\Re(s)\in ]0,\frac{1}{2}[$.\smallbreak
\noindent 
Finally, for  $s\in\C$, $\Re(s)>0$
\begin{equation*}
	\int_0^{\infty}K_0(x)x^{s-1}\dx=2^{s-2}\Gamma\left(\frac{s}{2}\right)^2,
\end{equation*}
the integral being absolutely convergent for $\Re(s)>0$.\smallbreak
\noindent
\end{lemme}
\begin{preuve}
	The first formula can be found in \cite{integraltransforms}*{Page 21}.\smallbreak
\noindent
We prove only the second one, since the proof of the third is identical. By \cite{Cohen}*{Proposition 9.8.8} we have for $\Re(s)\in ]0;1[$
\begin{equation}\label{MellinY0}
	\int_0^{\infty}Y_0(x)x^{s-1}\dx=-\frac{2^{s-1}}{\pi}\cos\left(\frac{\pi s}{2}\right)\Gamma\left(\frac{s}{2}\right)^2.
\end{equation}
Now we use
\begin{equation*}
	Y_0(x)\underset{x\rightarrow 0^+}{\sim}\frac{2}{\pi}\log(x)\text{ and } Y_0(x)\ll x^{-\frac{1}{2}},
\end{equation*}
(see \cite{Cohen}*{Proposition 9.8.1} and \cite{Cohen}*{Proposition 9.8.7}) which shows that the integral in \eqref{MellinY0} converges absolutely for $\Re(s) \in ]0,\frac{1}{2}[$ and that the formula extends to the range $\Re(s) \in ]0,\frac{3}{2}[$.
\end{preuve}
We now recall Parseval's theorem for Mellin transform. This result can be found in \cite{Indextransforms}*{Theorem 1.17}. 
\begin{proposition}\label{Parseval}
	Let $f,h:]0:+\infty[\rightarrow \R$ be two functions, we assume the Mellin transform of $f$; defined by $\mathcal{M}_f(s):=\int_0^{\infty}f(x)x^{s-1}\dx$ exists and is holomorphic in a strip $\{a<\Re(s)<b\}$. Then for any $c\in ]a;b[$ such that
\begin{equation*}
	t\mapsto\mathcal{M}_f(c+it)\in L^2(\R)\text{ and }\int_0^{\infty}x^{2-2c}\abs{h(x)}^2\frac{\dx}{x}<+\infty,
\end{equation*}
we have for $x>0$
\begin{equation*}
		\frac{1}{2i\pi}\int_{(c)}\mathcal{M}_f(s)\mathcal{M}_h(1-s)x^{-s}\ds=\int_0^{\infty}f(xt)h(t)\dt. 
	\end{equation*}
\end{proposition}
We now state some results regarding the Fourier coefficients of $f$. We start by recalling two well known and very deep results : Deligne's bound (deduces from \cite{Deligne}) and Rankin-Selberg's estimate (see \cite{Rankin} and \cite{Selberg}). 
\begin{thm}\label{Deligne}
We have the individual bound
\begin{equation*}
	\forall \eps>0,\quad \abs{a(n)}\ll n^{\eps}, 
\end{equation*}
and the Rankin-Selberg estimate
\begin{equation*}
	\sum_{n\leq x}a(n)^2=cx+O(x^{\frac{3}{5}}),
\end{equation*}
where $c=\Res\left(L(f\otimes f,s=1\right)$.
\end{thm}

\begin{lemme}\label{moment1fonctionL}
	The function $L\left(f\otimes f,s\right)=\sum_{n=1}^{\infty}a(n)^2n^{-s}$, defined for $\Re(s)>1$, can be analytically continued to a meromorphic function on $\Re(s)\geq \frac{1}{2}$ with only a simple pole at $s=1$. If we assume \eqref{hypo} for some $\delta\in[0;\frac{1}{12}]$, then $L(f\otimes f,s)$ can be analytically continued to a meromorphic function on $\Re(s)\geq \frac{1}{3}+2\delta$ with only a simple pole at $s=1$. Moreover, it satisfies for $\sigma\geq\frac{1}{2}$, $T,T'>1$ and $\eps>0$
\begin{equation*}
	\int_{T}^{2T}\int_{T'}^{2T'}\abs*{L\left(f\otimes f,\sigma+it+it'\right)}\dt\dt'
 \ll \min(T,T')\left((T+T')^{1+\eps}+(T+T')^{2-\frac{3}{2}\sigma+\eps}\right)
\end{equation*} 
Let $\delta\in[0;\frac{1}{12}]$. If we assume \eqref{hypo}, then for $T,T'>1$ and $\eps>0$ we have
\begin{equation}
	\int_{T}^{2T}\int_{T'}^{2T'}\abs*{L\left(f\otimes f,\frac{1}{3}+2\delta+it+it'\right)}\dt\dt'\ll \min(T,T')(T+T')^{\frac{5}{3}-5\delta+\eps}.
\end{equation}
\end{lemme}
\begin{preuve}
	A proof of analytic continuation can be found in \cite{Bump}*{Theorem 1.6.2}.\smallbreak
\noindent 
Let $\sigma\geq\frac{1}{2}$. We have,
\begin{equation*}
	\int_{t\sim T}\int_{t'\sim T'}\abs*{L\left(f\otimes f,\sigma+it+it'\right)}\dt\ll\int_{x\sim T+T'}\abs*{L\left(f\otimes f,\sigma+ix\right)}\min(T,T')\dx.
\end{equation*}
We can factorise $L(f\otimes f,\cdot)$ as the product of two $L$-functions. One has degree $1$ and the other has degree $3$. Precisely we have 
\begin{equation*}
	L(f\otimes f,s)=\frac{\zeta(s)}{\zeta(2s)}L(\text{sym}^2f,s),
\end{equation*}
where 
\begin{equation*}
	L(\text{sym}^2f,s)=\sum_{n=1}^{\infty}\frac{b(n)}{n^s},\quad b(n)=\sum_{ml^2=n}a(m^2).
\end{equation*}
The function $L(\text{sym}^2f,s)$ is the symmetric-square $L$-function (see \cite{Topics}*{Equation~(13.60)}). We have by the Cauchy-Schwarz inequality
\begin{equation*}\label{CS}
	\int_{T+T'}^{2(T+T')} \hspace{-0.8em}\abs*{L\left(f\otimes f,\sigma+ix\right)}\dx\leq \left(\int_{T+T'}^{2(T+T')}\abs*{\frac{\zeta(\sigma+ix)}{\zeta(2(\sigma+ix))}}^2\dx\right)^{\frac{1}{2}}\left(\int_{T+T'}^{2(T+T')}\hspace{-0.8em}\abs*{L(\text{sym}^2f,\sigma+ix)}^2\dx\right)^{\frac{1}{2}}.
\end{equation*}
The second moment of the $\zeta$ function on the critical line is well known (see \cite{Titch}*{Theorem 7.2}) and we have for $x\sim T+T'$
\begin{equation*}
	\frac{1}{\zeta(2\sigma+2ix)}\ll (T+T')^{\eps}. 
\end{equation*}
We now need to bound for an arbitrary $T>0$
\begin{equation*}
	\int_T^{2T}\abs*{L(\text{sym}^2f,\sigma+it)}^2\dt.
\end{equation*} 
	The approximate functional equation (see \cite{Iwakow}*{Theorem 5.3}) for $L$-functions gives for $t\in\R$
\begin{equation*}
	L(\text{sym}^2f,\sigma+it)=\sum_{n=1}^{\infty}\frac{b(n)}{n^{\sigma+it}}V_{\sigma+it}\left(\frac{n}{\sqrt{q}}\right)+\eps\left(f, \sigma+it\right)\sum_{n=1}^{\infty}\frac{b(n)}{n^{1-\sigma-it}}V_{\sigma-it}\left(\frac{n}{\sqrt{q}}\right),
\end{equation*}
where 
\begin{equation*}
	V_s(y)=\frac{1}{2i\pi}\int_{(3)}y^{-u}G(u)\frac{\gamma(f,s+u)}{\gamma(f,s)}\frac{\du}{u},\quad  G(u)=e^{u^2},
\end{equation*}
\begin{equation*}
	\gamma(f,s)=\pi^{-\frac{3}{2}s}\prod_{j=1}^3\Gamma\left(\frac{s+\kappa_j}{2}\right),
\end{equation*}
\begin{equation*}
\eps(f,s)=\eps(f)q(f)^{\frac{1}{2}-s}\frac{\gamma(f,1-s)}{\gamma(f,s)}
\end{equation*}
with $\abs{\eps(f)}=1$, $q=q(f)$ is the conductor of $L(\text{sym}^2f,s)$ and $\kappa_j\in\C$ are the local parameters of $L(\text{sym}^2f,s)$ at infinity.\smallbreak
\noindent Standard computations lead to 
\begin{align*}
	L(\text{sym}^2f,\sigma+it)&=\sumd{N}\sum_{n=1}^{\infty}\frac{b(n)}{n^{\sigma+it}}W\left(\frac{n}{N}\right)\\
	&+\eps(f,\sigma+it)\sumd{N}\sum_{n=1}^{\infty}\frac{b(n)}{n^{1-\sigma-it}}W\left(\frac{n}{N}\right)+O(\mathfrak{q}^{\frac{1-\sigma}{2}}T^{-1+\eps}+\eps(f,\sigma+it)\mathfrak{q}^{\frac{\sigma}{2}}T^{-1+\eps}),
\end{align*}
where $W$ is a smooth function with compact support in $[1;2]$, $N\ll \sqrt{\mathfrak{q}}^{1+\eps}$ and 
\begin{equation*}
	\mathfrak{q}=\prod_{j=1}^3(\abs{s+\kappa_j}+3)\asymp T^3
\end{equation*}
is the analytic conductor of $f$.
Expanding the square we get
\begin{align*}
	\abs*{L(\text{sym}^2f,\sigma+it)}^2&\ll\abs*{\sumd{N}\sum_{n=1}^{\infty}\frac{b(n)}{n^{\sigma+it}}W\left(\frac{n}{N}\right)}^2+ \mathfrak{q}^{1-\sigma}T^{-2+\eps}\\
	&\quad +\abs*{\eps(f,\sigma+it)}^2\left(\abs*{\sumd{N}\sum_{n=1}^{\infty}\frac{b(n)}{n^{1-\sigma-it}}W\left(\frac{n}{N}\right)}^2+ \mathfrak{q}^{\sigma}T^{-2+\eps}\right).
\end{align*}
Stirling's formula (Lemma \ref{Stirling}) gives
\begin{equation*}
	\abs*{\eps(f,\sigma+it)}^2\ll \mathfrak{q} ^{1-2\sigma}T^{\eps}.
\end{equation*}
We expand the square and take the integral on the first sum
\begin{align*}
	&\int_{T}^{2T}\abs*{\sumd{N}\sum_{n=1}^{\infty}\frac{b(n)}{n^{\sigma+it}}W\left(\frac{n}{N}\right)}^2\dt\\
	 &\ll T^{1+\eps}+\sumd{N,M} \sum_{\substack{n,m=1\\ \hspace{0.22cm} n\neq m\phantom{-}}}^{\infty}\frac{b(n)\overline{b(m)}}{n^{\sigma}m^{\sigma}\ln\left(\frac{n}{m}\right)}W\left(\frac{n}{N}\right)\overline{W\left(\frac{m}{M}\right)}\left(\left(\frac{m}{n}\right)^{2iT}-\left(\frac{m}{n}\right)^{iT}\right)\\
	&\ll T^{1+\eps}+ \mathfrak{q}^{1-\sigma+\eps},
\end{align*}
where we used $\sigma\geq \frac{1}{2}$ to bound the first term.\smallbreak
\noindent
Similarly we get
\begin{align*}
	&\int_{T}^{2T}\abs*{\sumd{N}\sum_{n=1}^{\infty}\frac{b(n)}{n^{1-\sigma-it}}W\left(\frac{n}{N}\right)}^2\dt\\
	&\ll T \mathfrak{q}^{\sigma-\frac{1}{2}}+ \sumd{N,M}\sum_{\substack{n,m=1\\ \hspace{0.22cm} n\neq m\phantom{-}}}^{\infty} \frac{b(n)\overline{b(m)}}{(nm)^{1-\sigma}\ln\left(\frac{n}{m}\right)}W\left(\frac{n}{N}\right)\overline{W\left(\frac{m}{M}\right)}\left(\left(\frac{m}{n}\right)^{2iT}-\left(\frac{m}{n}\right)^{iT}\right)\\
	&\ll T \mathfrak{q}^{\sigma-\frac{1}{2}}+ \mathfrak{q}^{\sigma+\eps}.
\end{align*}
We finally have
\begin{align*}
	\int_{T}^{2T}\abs*{L(\text{sym}^2f,\sigma+it)}^2\dt &\ll T^{1+\eps}+\mathfrak{q}^{1-\sigma+\eps}+ \mathfrak{q} ^{1-\sigma+\eps}T^{-2}+ \mathfrak{q} ^{1-2\sigma}T^{\eps}(T \mathfrak{q}^{\sigma-\frac{1}{2}}+ \mathfrak{q}^{\sigma+\eps}+ \mathfrak{q}^{\sigma}T^{-2})\\
	&\ll \left(T+\mathfrak{q}^{1-\sigma}\right)T^{\eps}\\
	&\ll \left(T+T^{3-3\sigma}\right)T^{\eps},
\end{align*}
where we used $\sigma\geq \frac{1}{2}$ and $\mathfrak{q}\asymp T^3$.
We now use \cite{Titch}*{Theorem 7.2} to get
\begin{equation*}
	\int_T^{2T}\abs{\zeta(\sigma+it)}^2\dt\ll T 
\end{equation*}
together with \eqref{CS} this yields
\begin{equation*}
	\int_{T}^{2T}\abs{L(f\otimes f,\sigma+it)}\dt\ll T^{1+\eps}+T^{2-\frac{3}{2}\sigma+\eps}.
\end{equation*}
If we assume \eqref{hypo}, we use the functional equation for $L(\text{sym}^2f,s)$ and $\zeta(s)$ (see \cite{Iwakow}*{Equation~(5.4)}) and Stirling's formula (Lemma \ref{Stirling})
\begin{equation*}
	\abs{\zeta(\sigma+it)}\asymp \abs{t}^{\frac{1}{2}-\sigma}\abs{\zeta(1-\sigma-it)}
\end{equation*}
\begin{equation*}
	\abs{L(\text{sym}^2f,\sigma+it)}\asymp \abs{t}^{3(\frac{1}{2}-\sigma)}\abs{L(\text{sym}^2f,1-\sigma-it)}
\end{equation*}
We then have the bounds
\begin{equation*}
	\int_{0}^T\abs{\zeta(\sigma+it)}^2\dt\ll T ^{\frac{3}{2}-\sigma}\text{ and } \frac{1}{\zeta(2\sigma+2it)}\ll \abs{t}^{\eps},
\end{equation*}
(see \cite{Titch}*{Equation~(14.2.6)}), 
hence
\begin{align*}
	\int_{T}^{2T}\abs{L(f\otimes f,\sigma+it)}\dt &\ll T^{1-\sigma}\times T^{3(\frac{1}{2}-\sigma)}\times (T^{\frac{1}{2}+\eps}+T^{\frac{3}{2}\sigma+\eps})\\
	&\ll T^{4(\frac{1}{2}-\sigma)}\left(T^{1+\eps}+T^{\frac{1}{2}+\frac{3}{2}\sigma+\eps}\right), 
\end{align*}
which ends the proof. 
\end{preuve}  
The next Lemma is a result on the fourth moment of the $\zeta$ function (see \cite{Ivic}*{Equation (2.1)}). It will be used in Section \ref{divisor}.
\begin{lemme}\label{moment4zeta}
	We have for $T\geq 1$ and $\eps>0$
\begin{equation*}
	\int_0^T\abs*{\zeta\left(\frac{1}{2}+it\right)}^4\dt\ll_{\eps} T^{1+\eps}.
\end{equation*}
\end{lemme}

We Recall the Voronoi summation formula. For modular forms, the proof can be found in \cite{Iwakow}*{Equation (4.17)}. For the divisor function, the proof can be found in \cite{Bilinear}.
\begin{proposition}\label{Voronoi}
	Let $w:]0;+\infty[\rightarrow\C$ be a smooth function with compact support, then for $q\geq 1$ and $h\in(\Z/q\Z)^*$
\begin{equation*}
	\sum_{n=1}^{\infty}a(n)e\left(\frac{hn}{q}\right)w(n)=\frac{1}{q}\sum_{n=1}^{\infty}a(n)e\left(-\frac{\overline{h}n}{q}\right)\omega_J\left(\frac{\sqrt{n}}{q}\right),
\end{equation*}
where $h\overline{h}=1[q]$ and 
\begin{equation}\label{defomegaJ}
	\omega_J(\alpha)=2\pi i^k\int_0^{\infty}w(x)J_{k-1}(4\pi\alpha\sqrt{x})\dx.
\end{equation}
A similar relation holds for the divisor function $\tau(n)$
\begin{align*}
	\sum_{n=1}^{\infty}\tau(n)e\left(\frac{hn}{q}\right)w(n)&=q^{-1}\int_0^{\infty}(\log(x)+2\gamma-2\log(q))w(x)\dx+\sum_{n=1}^{\infty}\tau(n)e\left(-\frac{\overline{h}n}{q}\right)\omega_Y\left(\frac{\sqrt{n}}{q}\right)\\
	&\quad +\sum_{n=1}^{\infty}\tau(n)e\left(\frac{\overline{h}n}{q}\right)\omega_K\left(\frac{\sqrt{n}}{q}\right),
\end{align*}
where
\begin{equation}\label{omegaBdef}
	\omega_{B}(\alpha)=c_{B_0}\int_0^{\infty}w(x)B_0(4\pi\alpha\sqrt{x})\dx,
\end{equation}
with $B\in\{Y,K\}$, $c_{Y_0}=-2\pi$ and $c_{K_0}=4$. 
\end{proposition}
We will also need the next proposition, to handle the shifted convolution problems that appear in Section \ref{termeerreur} and Section \ref{termeerreurtau}. 
\begin{proposition}\label{problemconv}
	Let $u=(a(n))_{n\geq 1}$ (as in Theorem \ref{moduleqqc}) or $u=(\tau(n))_{n\geq 1}$ (as in Theorem \ref{divisorAP}). Let $g:\R^2\rightarrow\R$ be a smooth function with compact support in a box $[M;2M]\times [N;2N]$, for some $N,M\geq 1$. We assume $g$ satisfies 
\begin{itemize}
	\item $g(x,y)\ll A$,
	\item $\forall (i,j)\in\N^2,i+j\geq 1,\quad $ $\displaystyle x^{-i}y^{-j}\frac{\partial^{i+j}g}{\partial x^{i}\partial y^j}\ll CP^{i+j}$,
\end{itemize}
for some $A,C,P>0$. Then for every $h\in\Z^*$, we have
\begin{equation*}
	\sum_{n\mp m=h}u(n)u(m)g(m,n)=\delta_{u=\tau}\int_0^{\infty}g(x,\pm x\mp h)\Lambda_h(x,\pm x\mp h)\dx+O\left(AP^{\frac{5}{4}}(M+N)^{\frac{1}{4}}(MN)^{\frac{1}{4}+\eps}\right), 
\end{equation*}
where 
\begin{equation*}
	\Lambda_h(x,y)=\sum_{k=1}^{\infty}\frac{1}{k^2}c_k(h)(\log(x)-2\gamma-2\log(k))(\log(y)-2\gamma-2\log(k)),
\end{equation*}
where the implied constant is independent of $h$ and $g$. 
\end{proposition}
\begin{preuve}
We follow the proof of \cite{quadraticdivisorproblem}*{Theorem 1} and especially section 7. We follow their notations, we have here $a=b=1$, $X=M$ and $Y=N$. We focus on the parts of the proof that change. The function $E$ satisfies 
\begin{equation*}
	E^{(i,j)}\ll \frac{\max(A,B)}{qQ}\left(\frac{1}{qQ}\right)^{i+j},
\end{equation*}
for $(i,j)\in\N^2$. The length of sums in $m$ and $n$ is unchanged, that is, after application of the Voronoi summation formula, both sums satisfy
\begin{equation*}
	m\ll \frac{M}{Q^2}Q^{\eps}\text{ and }n\ll \frac{N}{Q^2}Q^{\eps}.
\end{equation*}
Then, \cite{quadraticdivisorproblem}*{Equation (30)} should be replaced by
\begin{equation*}
	\iint \ll A\min(M,N)(M+N)^{-1}\log(Q). 
\end{equation*}
The end of the proof is the same and we get our desired result. 	
\end{preuve}

\subsection{Special function in the Voronoi summation formula}\label{specialfunction}

As in \cite{LauZhao}, it is crucial to obtain bounds on the function $\omega$ which appears in the Voronoi summation formula. We recall here some notations and results from \cite{LauZhao} and establish new ones needed for the proof.

Let $w$ be a smooth function supported on $[H;X]$ satisfying $w(x)=1$ for all $x\in [2H;X-H]$ and $w^{(j)}(x)\ll_j H^{-j}$, where $q<H<\frac{X}{3}$ is a parameter, which will be optimized at the end of the paper.\smallbreak
\noindent We define $\psi$ as the Mellin transform of $w$ \textit{i.e} 
\begin{equation*}
	\forall s\in\C,\quad \psi(s)=\mathcal{M}_{w}(s)=\int_0^{\infty}w(x)x^{s-1}\dx.
\end{equation*}
\begin{lemme}\label{majpsi}
	Let $s=c+it$, $c\leq 1$ and $j\geq 1$, then
\begin{equation*}
	\psi(1-s)\ll_j HX^{-c}\left(\frac{XH^{-1}}{\abs{s}}\right)^j,
\end{equation*}
for $\abs{t}\geq 1$. If $c$ is not an integer we have
\begin{equation*}
	\psi(1-s)\ll_j HX^{-c}\left(\frac{XH^{-1}}{1+\abs{t}}\right)^{j},
\end{equation*}
uniformly for $t\in\R$. 
\end{lemme}
\begin{preuve}
Recall that, 
\begin{equation*}
	\psi(1-s)= \int_0^{\infty}w(x)x^{-s}\dx.
\end{equation*}
We integrate by parts $j$ times to obtain
\begin{equation*}
	\psi(1-s)=\frac{1}{(s-1)\dots(s-j)}\int_0^{\infty}w^{(j)}(x)x^{-s+j}\dx. 
\end{equation*}
We then use the bound $w^{(j)}\ll H^{-j}$ and that $w^{(j)}$ is supported on $[H;2H]\cup [X-H;X]$ to get our result.   
\end{preuve}

The next proposition will give an expression of the functions $\omega_B$ for $B\in\{J,Y,K\}$ as an inverse Mellin transform. 
\begin{proposition}\label{propomega}
	Let $n\geq 1$ be an integer and  $\alpha>0$ be a real number, we have the following formula
\begin{equation}\label{propomegaJ}
	\omega_J\left(\sqrt{\alpha}\right)=\frac{1}{2i\pi}\int_{(\sigma)}F(s)\psi(1-s)\alpha^{-s}\ds,
\end{equation}
where $\displaystyle\sigma>-\frac{k-1}{2}$, 
\begin{equation*}
	F(s)=2\pi i^k(2\pi)^{-2s}\frac{\Gamma\left(\frac{k-1}{2}+s\right)}{\Gamma\left(\frac{k+1}{2}-s\right)},
\end{equation*}
\begin{equation}\label{propomegaY}
	\omega_Y\left(\sqrt{\alpha}\right)=\frac{1}{2i\pi}\int_{(\sigma)}G(s)\psi(1-s)\alpha^{-s}\ds,
\end{equation}
where $\sigma>0$, 
\begin{equation*}
	G(s)=2(2\pi)^{-2s}\cos(\pi s)\Gamma(s)^2,
\end{equation*}
\begin{equation}\label{propomegaK}
	\omega_K\left(\sqrt{\alpha}\right)=\frac{1}{2i\pi}\int_{(\sigma)}H(s)\psi(1-s)\alpha^{-s}\ds,
\end{equation}
where $\sigma>0$ and 
\begin{equation*}
	H(s)=(2\pi)^{-2s}\Gamma(s)^2. 
\end{equation*}
\end{proposition}
\begin{preuve}
We start with the proof of \eqref{propomegaJ}. We have by definition
\begin{equation*}
	\omega_J\left(\sqrt{\alpha}\right)=2\pi i^k\int_0^{\infty}w(x)J_{k-1}\left(4\pi\sqrt{\alpha x}\right)\dx.
\end{equation*}
We then use Lemma \ref{mellinbessel} to get for $\sigma\in]0,\frac{k-1}{2}[$
\begin{equation*}
	\omega_J\left(\sqrt{\alpha}\right)=2\pi i^{k}\int_0^{\infty}w(x)\frac{1}{2i\pi}\int_{(\sigma)}\frac{\Gamma(s)}{\Gamma(k-s)}\left(2\pi\sqrt{\alpha x}\right)^{k-1-2s}\ds. 
\end{equation*}
We make the change of variables $-2s' = k - 1 - 2s$, \textit{i.e.} $s' = s - \frac{k-1}{2}$, and we obtain
\begin{align*}
	\omega_J\left(\frac{\sqrt{n}}{r}\right)&=\frac{2\pi i^k}{2i\pi}\int_{(\sigma')}\int_0^{\infty}w(x)\frac{\Gamma\left(\frac{k-1}{2}+s\right)}{\Gamma\left(\frac{k+1}{2}-s\right)}\left(2\pi\sqrt{\alpha x}\right)^{-2s}\ds\\
	&=\frac{2\pi i^k}{2i\pi}\int_{(\sigma')} \frac{\Gamma\left(\frac{k-1}{2}+s\right)}{\Gamma\left(\frac{k+1}{2}-s\right)}(2\pi)^{-2s}\alpha^{-s}\left(\int_0^{\infty}w(x)x^{-s}\dx\right)\ds\\
	&=\frac{1}{2i\pi}\int_{(\sigma')}F(s)\psi(1-s)\alpha^{-s}\ds,
\end{align*} 	
where $\sigma':=\sigma-\frac{k-1}{2}$ satisfies
\begin{equation*}
	\sigma'\in \left]-\frac{k-1}{2};0\right[. 
\end{equation*}
We can swap the two integrals because 
\begin{equation*}
	\int_{\R}\int_{H}^X\frac{\abs{w(x)}}{x^{2\sigma'}}\abs{F(\sigma+it)}\dx\dt\ll\int_{\R}\int_{H}^X \frac{\abs{w(x)}}{x^{2\sigma'}}\frac{1}{(1+\abs{t})^{1-2\sigma'}}\dt\dx<\infty, 
\end{equation*}
where we used Lemma \ref{estimF}. 
\smallbreak
\noindent We now need to justify why we can take any $\sigma'>-\frac{k-1}{2}$ in the contour integral.
The function 
\begin{equation*}
	s\mapsto F(s)\psi(1-s)\alpha^{-s}
\end{equation*}
is holomorphic in the region $\left\{\Re(s)>-\frac{k-1}{2}\right\}$. Using the fact that $F(\sigma'+it)$ has polynomial growth in $t$ and Lemma \ref{majpsi}, we can shift the contour to the right and hence the last equality holds for any $\sigma'>-\frac{k-1}{2}$.\smallbreak
\noindent
We now prove \eqref{propomegaY}. We have by Proposition \ref{mellinbessel} and a change of variable
\begin{equation*}
	\int_0^{\infty}Y_0(4\pi\sqrt{\alpha x})x^{s-1}\dx=-\frac{2}{(4\pi)^{2s}\alpha^{s}}\frac{2^{2s-1}}{\pi}\cos(\pi s)\Gamma(s)^2=-\frac{1}{\pi(2\pi)^{2s}\alpha ^s}\cos(\pi s)\Gamma(s)^2. 
\end{equation*}
Hence by Mellin inversion
\begin{equation*}
	\omega_Y(\sqrt{\alpha})=-2\pi\int_0^{\infty}w(x)Y_0(4\pi\sqrt{\alpha x})\dx=-2\pi\int_0^{\infty}\left(\frac{1}{2i\pi}\int_{(\sigma)}\psi(1-s)x^{s-1}\ds\right)Y_0(4\pi\sqrt{\alpha x})\dx.
\end{equation*} 
For $\sigma\in]0;\frac{1}{2}[$, both integrals are absolutely convergent hence we can swap them to get
\begin{align*}
	\omega_Y(\sqrt{\alpha})&=-\frac{2\pi}{2i\pi}\int_{(\sigma)}\psi(1-s)\int_0^{\infty}Y_0(4\pi\sqrt{\alpha x})x^{s-1}\dx\ds\\
	&=\frac{1}{2i\pi}\int_{(\sigma)}2\cos(\pi s)\Gamma(s)^2(2\pi)^{-2s}\psi(1-s)\alpha^{-s}\ds.
\end{align*}
Again, by the same argument as for $\omega_J$, we may take any $\sigma>0$ in the integral.

\noindent The proof of \eqref{propomegaK} is similar.

\end{preuve}
We collect in the next lemma every bound we need for $\omega_B$ and its derivatives.
\begin{lemme}\label{majomega}
	Let $x>0$ and let $h\geq 1$ be an integer, then for $B\in\{J,Y,K\}$
\begin{enumerate}
	\item $ \forall j\geq 1$, $\displaystyle \omega_B\left(\frac{\sqrt{x}}{h}\right)\ll h^{\frac{1}{2}}n^{-\frac{1}{4}}X^{-\frac{1}{4}}H\left(\frac{X^{\frac{1}{2}}h}{\sqrt{n}H}\right)^j $.
	\item $\displaystyle \omega_K\left(\frac{\sqrt{x}}{h}\right)\ll h^{\frac{1}{2}}x^{-\frac{1}{4}}X^{\frac{3}{4}}\exp\left(-\frac{\sqrt{xH}}{h}\right)$.
	\item Let $Z>0$. There exists $C=C(B,Z,H,X)>0$ such that for all $j\geq 1$ and $\alpha\sim Z$
\begin{equation*}
	\omega_{B}^{(j)}(\alpha)\ll_j \alpha^{-j}CP^j,
\end{equation*}
where $P=XH^{-1}$.
	\item Let $Z>0$. There exists $C'>0$ such that for all $j\geq 1$ and $x\sim Z$
\begin{equation*}
	\left(\omega_{B}\left(\frac{\sqrt{x}}{h}\right)\right)^{(j)}\ll_j x^{-j}C'P^j,
\end{equation*}
where $P=XH^{-1}$.
\end{enumerate}
\end{lemme}
\begin{preuve}
1. This is \cite{LauZhao}*{Lemma 3.1 $b)$}.\smallbreak
\noindent
2. This is \cite{LauZhao}*{Lemma 3.1 $d)$}.\smallbreak
\noindent
3. We only give details for $B=J$, but the same holds for $B=Y$ and $B=K$. Using Proposition \ref{propomega} we have
\begin{equation*}
	\forall \alpha>0,\quad \omega_J(\alpha)=\frac{1}{2i\pi}\int_{(c)}F(s)\psi(1-s)\frac{\ds}{\alpha^s}.
\end{equation*}
We now differentiate this expression $j$ times
\begin{equation*}
	\omega_J^{(j)}(\alpha)=\frac{1}{2i\pi}\int_{(c)}F(s)\psi(1-s)\frac{(-1)^{j}s(s+1)\dots(s+j-1)}{\alpha^{s+j}}\ds.
\end{equation*}
Lemma \ref{estimF} shows that $F(c+it)$ grows at most polynomially in $t$, that is
\begin{equation*}
	F(c+it)\ll (1+\abs{t})^{l},
\end{equation*}
for some  integer $l\geq 1$. We then have, 
\begin{equation*}
	\omega_J^{(j)}(\alpha)\ll \alpha^{-j}\alpha^{-c}\int_{\R}(1+\abs{t})^l\abs{\psi(1-s)}(1+\abs{t})^j\dt.
\end{equation*} 
 We then use Lemma \ref{majpsi} to write
\begin{equation*}
	\psi(1-s)\ll HX^{-c}\frac{(XH^{-1})^{j+l+2}}{(1+\abs{t})^{j+l+2}},
\end{equation*}
and finally
\begin{equation*}
	\omega_J^{(j)}(\alpha)\ll \alpha^{-j}\alpha^{-c}HX^{-c}(XH^{-1})^{l+2}\int_{\R}\frac{\dt}{(1+\abs{t})^2}\ll \alpha^{-j}CP^{j},
\end{equation*}
with $C=Z^{-c}HX^{-c}(XH^{-1})^{l+2}$.\smallbreak
\noindent
4. This is easy using the previous bound on $\omega_B^{(j)}(\alpha)$ and  \cite{Knightly}*{Proposition (8.7)}.
\end{preuve}
\begin{rmq}\label{sommelisse}
Let $w:\R\rightarrow\R$ be a smooth function with compact support in $[1;2]$, define the following Hankel transform
\begin{equation*}
	\omega(\alpha)=\int_0^{\infty}w\left(\frac{x}{X}\right)J_{k-1}\left(4\pi\alpha x\right)\dx, 
\end{equation*}
where $J_{k-1}$ is the Bessel function of the first kind. Using the same arguments as in the proof of the third part of Lemma \ref{majomega} we can show that, for any $j \geq 1$, one has
\begin{equation*}
	\omega^{(j)}(\alpha)\ll_j \alpha^{-j}\left(\frac{X}{\alpha}\right)^{\frac{1}{2}}.
\end{equation*}
	We can compare this with the third part of Lemma \ref{majomega}. This is another reason why smooth sums are easier to deal with.  
\end{rmq}

The next two Lemmas will be used to obtain the main term respectively in $A(X,q)$ and $A(X,q;\tau)$.
\begin{lemme}\label{persevaltp}
	One has for any $\sigma\in\R$
\begin{equation*}
	\frac{1}{2i\pi}\int_{(\sigma)}F(s)\psi(s)F(1-s)\psi(1-s)\ds=X+O(H).
\end{equation*}
\end{lemme}
\begin{proof}
	It is easy to see that 
\begin{equation*}
	F(s)F(1-s)=1
\end{equation*}	
so
\begin{equation*}
	\frac{1}{2i\pi}\int_{(c)}F(s)\psi(s)F(1-s)\psi(1-s)\ds=\frac{1}{2i\pi}\int_{(c)}\psi(s)\psi(1-s)\ds.
\end{equation*}
Parseval's formula (Proposition \ref{Parseval}) gives
\begin{equation*}
	\frac{1}{2i\pi}\int_{(c)}\psi(s)\psi(1-s)\ds=\int_0^{\infty}w(x)^2\dx.
\end{equation*}
Since $w$ is compactly supported on $[H;X]$ and satisfies $w\equiv 1$ on $[2H;X-H]$ we immediately get
\begin{equation*}
	\int_0^{\infty}w(x)^2\dx=X+O(H),
\end{equation*} 
which concludes the proof. 
\end{proof}
\begin{lemme}\label{tptau}
	Let $r\geq 1$ and $Q(\xi)$ be a polynomial in $\log(\xi)$ of degree $m$. Then for $B=Y$ or $K$, we have
\begin{equation*}
	\int_0^{\infty}Q(\xi r^2)\omega_{B}(\sqrt{\xi})^2\d\xi=X\widetilde{Q_B}\left(\frac{r^2}{X}\right)+O(X^{\frac{3}{4}}H^{\frac{1}{4}}\log(Xr)^m),
\end{equation*}
where $\widetilde{Q_B}(\xi)$ is a polynomial in $\log(\xi)$ of degree $m$ whose leading coefficients has the same sign as $Q$. 
\end{lemme}
\begin{preuve}
This is \cite{LauZhao}*{Lemma 3.5}. 	
\end{preuve}

\section{Variance of Fourier coefficients in arithmetic progressions}\label{proofTh}
This section is devoted to the proof of Theorem \ref{moduleqqc}. We first smooth out the sum and then split the resulting smooth sum as
\begin{equation*}
	\mathrm{MT}(X,q)+E(X,q);
\end{equation*}
where $\mathrm{MT}(X,q)$ accounts for the diagonal terms and $E(X,q)$ is the off-diagonal contribution 

We will extract the main term in $\mathrm{MT}(X,q)$ using the inverse Mellin transform and bound $E(X,q)$ using Proposition \ref{problemconv}.
\subsection{Smoothing out the sum}
As in \cite{LauZhao}, the first step is to smooth out the sum $A(X,q)$ and then apply Voronoi's summation formula. 
\begin{proposition}\label{Awsplit}
We can split $A(X,q)$ in the following way
\begin{equation*}
	A(X,q)=A_w(q;a)+O(A_w(q;a)^{\frac{1}{2}}Hq^{-\frac{1}{2}}+H^2q^{-1}),
\end{equation*}
where 
\begin{equation*}
	A_w(q;a):=\sum_{b(q)}\abs*{\sum_{n\equiv b[q]}a(n)w(n)}^2=\mathrm{MT}(X,q)+E(X,q)+O(X^{-2026})
\end{equation*}
with 
\begin{equation*}
	\mathrm{MT}(X,q)= \frac{1}{q}\sum_{r\mid q}\frac{\phi(r)}{r^2}\sum_{1\leq n\leq r^2XH^{-2}X^{\eps}}a(n)^2\omega_J\left(\frac{\sqrt{n}}{r}\right)^2
\end{equation*}
and
\begin{equation*}
	E(X,q)=\frac{1}{q}\sum_{dr\mid q}\frac{\mu(d)}{d^2r}\sum_{\substack{1\leq n,m\leq d^2r^2XH^{-2}X^{\eps}\\ n\equiv m[r]\\
	n\neq m\phantom{-}}} a(n)a(m)\omega_J\left(\frac{\sqrt{n}}{dr}\right)\omega_J\left(\frac{\sqrt{m}}{dr}\right).
\end{equation*} 

\end{proposition}
\begin{preuve}
Recall the definition of $A(X,q)$
\begin{equation*}
	A(X,q)=\sum_{b(q)}\abs*{\sum_{\substack{n\leq X\\ \hspace{0.22cm} n\equiv b[q]\phantom{-}}}a(n)}^2.
\end{equation*}
By \cite{LauZhao}*{Propositon 4.1} we can write
\begin{equation*}
	A(X,q)=A_w(q;a)+O\left(A_w(q;a)^{\frac{1}{2}}E^{\frac{1}{2}}+E\right),
\end{equation*}
where 
\begin{equation*}
	E=\sum_{b(q)}\abs*{\sum_{\substack{H\leq n\leq 2H\\ \hspace{0.22cm} n\equiv b[q]\phantom{-}}} a(n)(1-w(n))+ \sum_{\substack{X-H\leq n\leq X\\ \hspace{0.22cm} n\equiv b[q]\phantom{-}}}a(n)(1-w(n))}^2.
\end{equation*}
By Cauchy-Schwarz's inequality and Rankin-Selberg estimate (Theorem \ref{Deligne}) we get,
\begin{equation*}
	E\ll H^2q^{-1},
\end{equation*}
assuming $H\geq X^{\frac{3}{5}}$.\smallbreak
\noindent
We now need to study $A_w(q;a)$. Using additive characters and Voronoi's summation formula as in \cite{LauZhao}*{Proposition 4.1} we get
\begin{equation*}
	A_w(q;a)=\frac{1}{q}\sum_{r\mid q}\sums{h(r)}\frac{1}{r^2}\sum_{1\leq n,m\leq r^2XH^{-2}X^{\eps}}\hspace{-0.3cm}a(n)a(m)\omega_J\left(\frac{\sqrt{n}}{r}\right)e\left(\frac{h(n-m)}{r}\right)\omega_J\left(\frac{\sqrt{m}}{r}\right)+O(X^{-2026}).
\end{equation*}
We separate the sums over $n$ and $m$ according to whether $n = m$ or not. In the sum where $n\neq m$ we use Möbius inversion to remove the condition $(h,r)=1$. We denote by $\mathrm{MT}(X,q)$ the sum corresponding to $n=m$, that is
\begin{equation*}
	\text{MT}(X,q)=\frac{1}{q}\sum_{r\mid q}\sums{h(r)}\frac{1}{r^2}\sum_{n\leq r^2XH^{-2}X^{\eps}}a(n)^2\omega_J\left(\frac{\sqrt{n}}{r}\right)^2=\frac{1}{q}\sum_{r\mid q}\frac{\phi(r)}{r^2}\sum_{n\leq r^2XH^{-2}X^{\eps}}a(n)^2\omega_J\left(\frac{\sqrt{n}}{r}\right)^2.
\end{equation*}
We denote $E(X,q)$ the sum where $n\neq m$, that is
\begin{equation*}
	E(X,q)=\frac{1}{q}\sum_{dr\mid q}\frac{\mu(d)}{d^2r^2}\sum_{h(r)}\sum_{\substack{1\leq n,m\leq d^2r^2XH^{-2}X^{\eps}\\
	n\neq m\phantom{-}}} a(n)a(m)\omega_J\left(\frac{\sqrt{n}}{dr}\right)\omega_J\left(\frac{\sqrt{m}}{dr}\right)e\left(\frac{h(n-m)}{r}\right),
\end{equation*}
executing the $h$-sum ends the proof. 
\end{preuve}

\begin{rmq}\label{q1/2}
In the next subsection, we will show that $A_w(q;a)\asymp X$, hence the error term $O\left(A_w(q;a)^{\frac{1}{2}}E^{\frac{1}{2}}\right)$ is at least $O\left(q^{\frac{1}{2}}X^{\frac{1}{2}}\right)$.
\end{rmq}

We will now study $\mathrm{MT}(X,q)$ and $E(X,q)$ separately in the next two subsections.   
\subsection{Main Term}\label{Mainterm}
This section is devoted to the proof of the following statement. 
\begin{proposition}\label{etudeTP}
We have 
\begin{equation*}
	\mathrm{MT}(X,q)=cX+O\left(q^{-1}X^{\frac{3}{2}}g(q)\right)+O(H),
\end{equation*}
where $c=\Res\left(L\left(f\otimes f,s\right),s=1\right)$.
\end{proposition}

\noindent By Lemma \ref{majomega} we have for any $j\geq 1$
\begin{equation*}
	\omega_J\left(\frac{\sqrt{n}}{r}\right)\ll r^{\frac{1}{2}}n^{-\frac{1}{4}}X^{-\frac{1}{4}}H\left(\frac{X^{\frac{1}{2}}r}{\sqrt{n}H}\right)^j,
\end{equation*}
hence
\begin{equation*}
	\omega_J\left(\frac{\sqrt{n}}{r}\right)\ll X^{-2026}
\end{equation*}
for $n\gg r^2XH^{-2}X^{\eps}$. We can complete the sum $\mathrm{MT}(X,q)$ of Proposition \ref{Awsplit} and write
\begin{equation*}
	\mathrm{MT}(X,q)=\frac{1}{q}\sum_{r\mid q}\frac{\phi(r)}{r^2}\sum_{n=1}^{\infty}a(n)^2\omega_J\left(\frac{\sqrt{n}}{r}\right)^2+O(X^{-2026}).
\end{equation*}
We first study 
\begin{equation*}
	\sum_{n=1}^{\infty}a(n)^2\omega_J\left(\frac{\sqrt{n}}{r}\right)^2. 
\end{equation*}
\begin{proposition}\label{suma(n)^2}
	We have 
\begin{equation*}
	\sum_{n=1}^{\infty}\label{TPqqc} a(n)^2\omega_J\left(\frac{\sqrt{n}}{r}\right)^2=cr^2X+O(r^2H)+O(rX^{\frac{3}{2}}),
\end{equation*}
where $c=\Res(L\left(f\otimes f,s+s'\right),s=1)$. 
\end{proposition}

\begin{preuve}
Using Proposition \ref{propomega}, we can write
\begin{equation*}
	\sum_{n=1}^{\infty}a(n)^2\omega_J\left(\frac{\sqrt{n}}{r}\right)^2=\sum_{n=1}^{\infty}a(n)^2\frac{1}{(2i\pi)^2}\int_{(\sigma)}\int_{(\sigma')}F(s)F(s')r^{2s+2s'}n^{-(s+s')}\psi(1-s)\psi(1-s')\ds\ds',
\end{equation*}
where $\sigma,\sigma'\in\left]-\frac{k-1}{2};+\infty\right[$. 
We choose $\sigma,\sigma'$ such that $\sigma + \sigma' > 1$; hence the sum and both integrals are absolutely convergent and we can swap sums and integrals. We obtain
\begin{equation}\label{doubleint}
	\sum_{n=1}^{\infty}a(n)^2\omega_J\left(\frac{\sqrt{n}}{r}\right)^2= \frac{1}{(2i\pi)^2}\int_{(\sigma)}\int_{(\sigma')}F(s)F(s')r^{2s+2s'}L\left(f\otimes f,s+s'\right)\psi(1-s)\psi(1-s')\ds\ds'.
\end{equation}
Let $\sigma=1+\eps$. We consider only the $s'$-integral \textit{i.e.}
\begin{equation*}
I:=\int_{(\sigma')}F(s')r^{2s'}L\left(f\otimes f,s+s'\right)\psi(1-s')\ds'. 	
\end{equation*}
We shift the $s'$-integral to $\sigma'=-\frac{1}{4}$ and we encounter only a simple pole at $s'=1-s$. The residue theorem gives us
\begin{equation*}
	I=2i\pi F(1-s)\times r^{2(1-s)}\times c\times\psi(s)+\int_{(-\frac{1}{4})}F(s')r^{2s'}L\left(f\otimes f,s+s'\right)\psi(1-s')\ds'.
\end{equation*}
We insert the previous relation in \eqref{doubleint} to get  
\begin{equation}\label{suma(n)}
	\sum_{n=1}^{\infty}a(n)^2\omega_J\left(\frac{\sqrt{n}}{r}\right)^2= cr^2\frac{1}{2i\pi}\int_{(\sigma)}F(s)F(1-s)\psi(s)\psi(1-s)\ds+E.
\end{equation}
The error term $E$ is defined by
\begin{equation*}
	E=\frac{1}{(2i\pi)^2}\int_{(\sigma)}\int_{\left(-\frac{1}{4}\right)}F(s)F(s')r^{2s+2s'}L\left(f\otimes f,s+s'\right)\psi(1-s)\psi(1-s')\ds\ds'.
\end{equation*}
We first study the first term on the right-hand side of \eqref{suma(n)}. Recall Lemma \ref{persevaltp}
\begin{equation*}
	\frac{1}{2i\pi}\int_{(\sigma)}F(s)\psi(s)F(1-s)\psi(1-s)\ds=X+O(H), 
\end{equation*}
so the first term on the right-hand side is
\begin{equation*}
	cr^2\left(X+O(H)\right)=cr^2X+O(r^2H).
\end{equation*}
We now turn to the error term $E$. We shift the $s$-integral to $\sigma=\frac{1}{4}$, encountering no poles. Next, we shift the $s'$-integral to $\sigma' = \frac{1}{4}$, again without encountering any poles. Finally, we have
\begin{equation}\label{erreurEdint}
	E\ll r \iint
	\abs{\mathcal{F}(t,t')}\abs*{L\left(f\otimes f,\frac{1}{2}+it+it'\right)}\dt\dt',
\end{equation}
where
\begin{equation*}
	\mathcal{F}(t,t')= F\left(\frac{1}{4}+it\right) F\left(\frac{1}{4}+it'\right) \psi\left(\frac{3}{4}-it\right) \psi\left(\frac{3}{4}-it'\right).
\end{equation*}
By Lemma \ref{majpsi}, both integrals are supported on
\begin{equation*}
	\abs{t},\abs{t'}\ll XH^{-1}X^{\eps}. 
\end{equation*}
We will show the double integral is $\ll X^{\frac{3}{2}}$. We split both integrals into dyadic pieces $t\sim T$ and $t'\sim T'$ where 
\begin{equation*}
	\abs{T},\abs{T'}\ll XH^{-1}X^{\eps}.
\end{equation*} 
We have for $t\sim T$ by Lemma \ref{estimF}
\begin{equation*}
	\abs*{F\left(\frac{1}{4}+it\right)}\asymp\frac{1}{(1+\abs{T})^{\frac{1}{2}}}
\end{equation*}
and by Lemma \ref{majpsi} with $k=1$
\begin{equation*}
	\abs*{\psi\left(\frac{3}{4}-it\right)}= \abs*{\psi\left(1-\left(\frac{1}{4}+it\right)\right)}\ll \frac{X^{\frac{3}{4}}}{1+\abs{T}}.
\end{equation*}
Thus, the double integral is
\begin{align*}
	&\ll X^{\frac{3}{2}}\left(1+\sumd{1\leq \abs{T},\abs{T'}\leq XH^{-1}X^{\eps}}(\abs{TT'})^{-\frac{3}{2}}\int_{t\sim T}\int_{t'\sim T'}\abs*{L\left(f\otimes f,\frac{1}{2}+it+it'\right)}\dt\dt'\right)\\
	&\ll X^{\frac{3}{2}}\left(1+ \sumd{1\leq \abs{T},\abs{T'}\leq XH^{-1}X^{\eps}}(\abs{TT}')^{-\frac{1}{2}}(\abs{T}+\abs{T}')^{\frac{1}{4}+\eps}\right)\\
	&\ll X^{\frac{3}{2}}\left(1+ \sumd{1\leq \abs{T},\abs{T'}\leq XH^{-1}X^{\eps}}\left(\frac{\abs{T}+\abs{T'}}{\abs{TT'}}\right)^{\frac{1}{2}}(\abs{TT'})^{-\frac{1}{4}+\eps}\right)\\
	&\ll X^{\frac{3}{2}},
\end{align*}
by Lemma \ref{moment1fonctionL}. We finally have
\begin{equation*}
	E\ll  rX^{\frac{3}{2}}
\end{equation*} 
and Proposition \ref{TPqqc} follows.
\end{preuve}

\begin{rmq}
Assume \eqref{hypo} for some $\delta\in[0;\frac{1}{12}]$, then if we shift the contour in \eqref{erreurEdint} to $\sigma=\sigma'=\frac{1}{6}+\delta$, and follow the same arguments as in the proof of Proposition \ref{suma(n)^2} we prove 
\begin{equation*}
	\sum_{n=1}^{\infty}\label{TPqqcbis} a(n)^2\omega_J\left(\frac{\sqrt{n}}{r}\right)^2=cr^2X+O(r^2H)+O(r^{\frac{2}{3}-4\delta}X^{\frac{5}{3}-2\delta}). 
\end{equation*}	
\end{rmq}

\begin{rmq}
Assume $s\mapsto L(f\otimes f,s)$ satisfies the Lindelöf hypothesis on average, that is 
\begin{equation*}
	\forall \sigma\in\left[0;\frac{1}{2}\right],\forall T>0,\forall \eps>0,\quad \int_0^T\abs*{L\left(f\otimes f,\sigma+it\right)}\dt\ll T^{4(\frac{1}{2}-\sigma)+1+\eps}. 
\end{equation*}
Let $\sigma\in \left[0;\frac{1}{4}\right]$. We shift both contours in \eqref{erreurEdint} to $\sigma$. Proceeding as before we get
\begin{equation*}
	E\ll r^{4\sigma}X^{2-2\sigma}\left(1+\sumd{\abs{T}\ll XH^{-1}X^{\eps}}\quad\sumd{\abs{T'}\ll XH^{-1}X^{\eps}}(\abs{TT'})^{2\sigma-1}(\abs{T}+\abs{T'})^{2-8\sigma+\eps}\right).
\end{equation*}
We see that the smallest value of $\sigma$ such that the double sum is $\ll X^{\eps}$ is 
\begin{equation*}
	\sigma=\frac{1}{6}. 
\end{equation*}
\end{rmq}

We can now use Proposition \ref{suma(n)^2} to prove Proposition \ref{etudeTP}.\smallbreak\noindent 
\begin{preuve}
Recall
\begin{equation*}
	\mathrm{MT}(X,q)=\frac{1}{q}\sum_{r\mid q}\frac{\phi(r)}{r^2}\sum_{n=1}^{\infty}a(n)^2\omega_J\left(\frac{\sqrt{n}}{r}\right)^2+O\left(X^{-2026}\right).
\end{equation*}
We use Proposition \ref{TPqqc} and we have 
\begin{equation*}
	\frac{1}{q}\sum_{r\mid q}\frac{\phi(r)}{r^2}\sum_{n=1}^{\infty}a(n)^2\omega\left(\frac{\sqrt{n}}{r}\right)^2= \frac{1}{q}\sum_{r\mid q}\frac{\phi(r)}{r^2}\left(cr^2X+O(r^2H)+O(rX^{\frac{3}{2}})\right)
\end{equation*}
and then 
\begin{equation*}
	\mathrm{MT}(X,q)=cX+O(H)+O\left(q^{-1}X^{\frac{3}{2}}g(q)\right),
\end{equation*}
which ends the proof. 
\end{preuve}

\begin{rmq}\label{TEdelta}
	If we assume \eqref{hypo} for some $\delta\in[0;\frac{1}{12}]$, we can compute the error term in a different way. We split the $r$-sum according to whether $r<X^{\frac{2}{3}}q^{-\frac{1}{3}}$ or not. If $r<X^{\frac{2}{3}}q^{-\frac{1}{3}}$ we use Lemma \ref{majomega} 1. with $j=1$ to get
\begin{equation*}
	\sum_{n=1}^{\infty}a(n)^2\omega_J\left(\frac{\sqrt{n}}{r}\right)^2\ll r^3X^{\frac{1}{2}}
\end{equation*}
and 
\begin{equation*}
	\frac{1}{q}\sum_{\substack{r\mid q\\ \hspace{0.22cm}r< X^{\frac{2}{3}}q^{-\frac{1}{3}}\phantom{-}}}\sum_{n=1}^{\infty}a(n)^2\omega_J\left(\frac{\sqrt{n}}{r}\right)^2\ll \frac{X^{\frac{1}{2}}}{q}\sum_{\substack{r\mid q\\ \hspace{0.22cm}r< X^{\frac{2}{3}}q^{-\frac{1}{3}}\phantom{-}}}r\phi(r)\ll \frac{X^{\frac{11}{6}}}{q^{\frac{5}{3}}}\tau(q)\ll \frac{X^{\frac{5}{3}}}{q^{\frac{4}{3}}}\tau(q). 
\end{equation*}
We now consider the divisors $r$ of $q$ such that $r>X^{\frac{2}{3}}q^{-\frac{1}{3}}$. One checks that for $\delta>0$
\begin{equation*}
	rX^{\frac{3}{2}}\leq r^{\frac{2}{3}+4\delta}X^{\frac{5}{3}-2\delta}\Longleftrightarrow r\leq X^{\frac{1}{2}}.
\end{equation*}
We write
\begin{align*}
	\mathrm{MT}(X,q)&=\frac{1}{q}\sum_{\substack{r\mid q\\ \hspace{0.22cm}r>X^{\frac{2}{3}}q^{-\frac{1}{3}}\phantom{-}}}\frac{\phi(r)}{r^2}\left(cr^2X+O\left(\min(rX^{\frac{3}{2}},r^{\frac{2}{3}+4\delta}X^{\frac{5}{3}-2\delta})\right)\right)+O\left(\frac{X^{\frac{5}{3}}}{q^{\frac{4}{3}}}\tau(q)\right)+O(H)\\
	&=cX-\frac{cX}{q}\sum_{\substack{r\mid q\\ \hspace{0.22cm}r\leq X^{\frac{2}{3}}q^{-\frac{1}{3}}\phantom{-}}}\phi(r)\\
	&\quad +O\left(\frac{X^{\frac{3}{2}}}{q} \sum_{\substack{r\mid q\\ \hspace{0.05cm}X^{\frac{2}{3}}q^{-\frac{1}{3}}<r<X^{\frac{1}{2}}\phantom{-}}}\frac{\phi(r)}{r}+\frac{X^{\frac{5}{3}}}{q}\left(\frac{q^2}{X}\right)^{2\delta} \sum_{\substack{r\mid q\\ \hspace{0.07cm}r\geq X^{\frac{1}{2}}\phantom{-}}}\frac{\phi(r)}{r^{\frac{4}{3}}}+\frac{X^{\frac{5}{3}}}{q^{\frac{4}{3}}}\tau(q)+H\right).
\end{align*}
We now use 
\begin{equation*}
\frac{cX}{q}\sum_{\substack{r\mid q\\ \hspace{0.22cm}r\leq X^{\frac{2}{3}}q^{-\frac{1}{3}}\phantom{-}}}\phi(r)\ll \frac{X^{\frac{5}{3}}}{q^{\frac{4}{3}}}\tau(q) 
\end{equation*}
and for $r\geq X^{\frac{1}{2}}$ with $r\mid q$
\begin{equation*}
\frac{1}{r^{\frac{4}{3}}}\ll \frac{1}{d_{\frac{1}{2}}(q)^{\frac{1}{3}}r},
\end{equation*}
and finally we have 
\begin{equation*}
	\mathrm{MT}(X,q)=cX+O(E_{\delta}+H),
\end{equation*}
where $E_{\delta}$ is defined by \eqref{MET}
\end{rmq}

\subsection{Error term}\label{termeerreur}
In this subsection, we study the error term $E(X,q)$, whose definition we recall below
\begin{equation*}
	E(X,q)=\frac{1}{q}\sum_{dr\mid q}\frac{\mu(d)}{d^2r}\sum_{\substack{1\leq n,m\leq d^2r^2XH^{-2}X^{\eps}\\ n\equiv m[r]\\
	n\neq m\phantom{-}}} a(n)a(m)\omega_J\left(\frac{\sqrt{n}}{dr}\right)\omega_J\left(\frac{\sqrt{m}}{dr}\right).
\end{equation*}
Precisely, we will prove the following.
\begin{proposition}\label{propE(X,q)}
	The following estimate holds for $E(X,q)$
\begin{equation*}
	E(X,q)\ll q^{\frac{5}{6}}X^{\frac{4}{3}}H^{-\frac{5}{4}}X^{\eps}.
\end{equation*}
\end{proposition}
\noindent The sum $E(X,q)$ is a \textit{shifted convolution sum}. Letting $h=dr$ and using a dyadic partition in both variables $n$ and $m$ we have 
\begin{align*}
	E(X,q)&=\frac{1}{q}\sum_{dr\mid q}\frac{\mu(d)}{d^2r}\sumd{M,N}\sum_{\substack{1\leq n,m\leq h^2XH^{-2}X^{\eps}\\ n\equiv m[r]\\ n\neq m\phantom{-}}}a(n)a(m)\omega_J\left(\frac{\sqrt{m}}{h}\right) \rho\left(\frac{m}{M}\right)\omega_J\left(\frac{\sqrt{n}}{h}\right) \rho\left(\frac{n}{N}\right)\\
	&=\frac{1}{q}\sum_{dr\mid q}\frac{\mu(d)}{d^2r}\sumd{M,N}\mathcal{E}(M,N,h,r),
\end{align*}
where $\rho$ is a smooth function with compact support in $[1;2]$ such that
\begin{equation*}
	\forall x\in\R,\quad \sumd{M}\rho\left(\frac{x}{M}\right)=\sum_{k\in\Z}\rho\left(\frac{x}{2^k}\right)=1. 
\end{equation*}
Both variables $M,N$ satisfy : $1\leq M,N\leq h^2XH^{-2}X^{\eps}$. We can assume without loss of generality that $M\leq N$.\smallbreak
\noindent
The trivial bound for $\mathcal{E}(M,N,h,r)$ is the following.
\begin{proposition}\label{bornetriviale}
We have 
\begin{equation*}
	\mathcal{E}(M,N,h,r)\ll h^3X^{\frac{1}{2}}r^{-1}M^{\frac{1}{4}}N^{\frac{1}{4}}X^{\eps}.
\end{equation*}
\end{proposition}
\begin{preuve}
We can write, using Lemma \ref{majomega}~1. with $j=1$
\begin{align*}
	\mathcal{E}(M,N,h,r)&=\sum_{\substack{\abs{l}\ll r^{-1}N\\ m\sim M\phantom{-}}}a(m)a(m+lr) \omega_J\left(\frac{\sqrt{m}}{h}\right) \rho\left(\frac{m}{M}\right)\omega_J\left(\frac{\sqrt{m+lr}}{h}\right) \rho\left(\frac{m+lr}{N}\right)\\
	&\ll \sum_{\substack{\abs{l}\ll r^{-1}N\\ m\sim M\phantom{-}}}\abs{a(m)a(m+lr)}X^{\frac{1}{2}}h^3m^{-\frac{3}{4}}(lr)^{-\frac{3}{4}}\\
	&\ll h^3X^{\frac{1}{2}}r^{-1}M^{\frac{1}{4}}N^{\frac{1}{4}}X^{\eps}.
\end{align*}
\end{preuve}
\begin{rmq}
	Note that using $M\leq N\leq h^2 XH^{-2}X^{\eps}$ we get
\begin{equation*}
	\mathcal{E}(M,N,h,r)\ll h^4r^{-1}XH^{-1}X^{\eps}
\end{equation*}
and 
\begin{equation*}
	E(X,q)\ll \frac{1}{q}\sum_{dr\mid q}\frac{d^4r^3}{d^2r}XH^{-1}X^{\eps}\ll qXH^{-1}X^{\eps},
\end{equation*}
which is the bound obtained in \cite{LauZhao}. 
\end{rmq}
We will obtain another bound for $\mathcal{E}(M,N,h,r)$ using Proposition \ref{problemconv}.
We can write
\begin{equation*}
	\mathcal{E}(M,N,h,r)=\sum_{l\leq 2r^{-1}N}\sum_{n-m=lr}a(m)a(n)g(m,n).
\end{equation*}
Where 
\begin{equation*}
	g(x,y)=\omega_J\left(\frac{\sqrt{x}}{h}\right)\rho\left(\frac{x}{M}\right)\omega_J\left(\frac{\sqrt{y}}{h}\right)\rho\left(\frac{y}{N}\right). 
\end{equation*}
Using Lemma \ref{majomega} we show that
\begin{equation*}
	g(x,y)\ll h^3X^{\frac{1}{2}}M^{-\frac{3}{4}}N^{-\frac{3}{4}}=:A
\end{equation*}
and for $(i,j)\in\N^2$, $i+j\geq 1$, one has 
\begin{equation*}
	\frac{\partial^{i+j} g}{\partial x^{i}\partial y^j}(x,y)\ll CP^{i+j},
\end{equation*}
with $P=XH^{-1}$.\smallbreak
\noindent
Using Proposition \ref{problemconv} we prove the following. 
\begin{proposition}\label{bornebis}
	We have
\begin{equation*}
	\mathcal{E}(M,N,h,r)\ll h^3r^{-1}X^{\frac{7}{4}}H^{-\frac{5}{4}}M^{-\frac{1}{2}}N^{\frac{3}{4}}(MN)^{\eps}. 
\end{equation*}
\end{proposition}
We can now prove Proposition \ref{propE(X,q)}.\smallbreak
\begin{preuve}
Let $\Omega > 0$ be a parameter to be optimized later.
If $M\leq\Omega$, we use the trivial bound (Proposition \ref{bornetriviale})
\begin{equation}\label{borne1}
	\sumd{M\leq\Omega}\sumd{N}\mathcal{E}(M,N,h,r)\ll \sumd{M\leq\Omega}\sumd{N}\\ h^3r^{-1}X^{\frac{1}{2}}M^{\frac{1}{4}}N^{\frac{1}{4}}X^{\eps}\ll h^{\frac{7}{2}}r^{-1}X^{\frac{3}{4}}H^{-\frac{1}{2}}\Omega^{\frac{1}{4}}X^{\eps}.
\end{equation}
If $M\geq\Omega$, we use Proposition \ref{bornebis}
\begin{equation}\label{borne2}
	\sumd{M\geq \Omega}\sumd{N}\mathcal{E}(M,N,h,r)\ll \sumd{M\geq\Omega}\sumd{N}h^3r^{-1}X^{\frac{7}{4}}H^{-\frac{5}{4}}M^{-\frac{1}{2}}N^{\frac{3}{4}}(MN)^{\eps} \ll h^{\frac{9}{2}}r^{-1}X^{\frac{5}{2}}H^{-\frac{11}{4}}\Omega^{-\frac{1}{2}}X^{\eps}.
\end{equation}
Recall
\begin{equation*}
	E(X,q)=\frac{1}{q}\sum_{dr\mid q}\frac{\mu(d)}{d^2r}\sumd{M,N}\mathcal{E}(M,N,dr,r).
\end{equation*}
Inserting \eqref{borne1} and \eqref{borne2} we get
\begin{align*}
	E(X,q)&\ll q^{-1}\sum_{dr\mid q}\left((dr)^{\frac{3}{2}}X^{\frac{3}{4}}H^{-\frac{1}{2}}\Omega^{\frac{1}{4}}X^{\eps}+(dr)^{\frac{5}{2}}X^{\frac{5}{2}}H^{-\frac{11}{4}}\Omega^{-\frac{1}{2}}X^{\eps}\right)\\
	&\ll q^{\frac{1}{2}}X^{\frac{3}{4}}H^{-\frac{1}{2}}\Omega^{\frac{1}{4}}X^{\eps}+ q^{\frac{3}{2}}X^{\frac{5}{2}}H^{-\frac{11}{4}}\Omega^{-\frac{1}{2}}X^{\eps}. 
\end{align*}
We now optimize both error terms by taking $\Omega = q^{\frac{4}{3}}X^{\frac{7}{3}}H^{-3}$, which yields
\begin{equation*}
	E(X,q)\ll q^{\frac{5}{6}}X^{\frac{4}{3}}H^{-\frac{5}{4}}X^{\eps}.
\end{equation*}
This is exaclty the desired error term. 
\end{preuve}

\subsection{Final computations}
We are now ready to prove Theorem \ref{moduleqqc}. One has by Proposition \ref{Awsplit}
\begin{align*}
	A(X,q)&=A_w(q;a)+O\left(A_w(q;a)^{\frac{1}{2}}Hq^{-\frac{1}{2}}+H^2q^{-1}\right)\\
	&=\text{MT}(X,q)+E(X,q) + O\left(A_w(q;a)^{\frac{1}{2}}Hq^{-\frac{1}{2}}+H^2q^{-1}\right)\\
	&=cX+O\left(q^{-1}X^{\frac{3}{2}}g(q)+H\right)+O\left(q^{\frac{5}{6}}X^{\frac{4}{3}}H^{-\frac{5}{4}}X^{\eps}\right)+O(X^{\frac{1}{2}}Hq^{-\frac{1}{2}} + H^2q^{-1}).
\end{align*}	
We optimize both error terms by choosing
\begin{equation*}
	H=\frac{1}{3}\max(q^{\frac{16}{27}}X^{\frac{10}{27}},q)\geq X^{\frac{3}{5}}.
\end{equation*}
and finally we obtain
\begin{equation*}
	A(X,q)=cX+O\left(q^{-1}X^{\frac{3}{2}}g(q)\right)+ O\left(q^{\frac{5}{54}}X^{\frac{47}{54}}X^{\eps} +q^{\frac{5}{27}}X^{\frac{20}{27}}+q^{\frac{1}{2}}X^{\frac{1}{2}}\right).
\end{equation*}
Now we see that
\begin{equation*}
	q^{\frac{5}{27}}X^{\frac{20}{27}}\leq q^{\frac{5}{54}}X^{\frac{47}{54}},
\end{equation*}
which proves Theorem \ref{moduleqqc}.
\begin{rmq}
Assume \eqref{hypo} for some $\delta\in[0;\frac{1}{12}]$. If we follow the same reasoning as in Remark \ref{TEdelta} we can replace $q^{-1}X^{\frac{3}{2}}g(q)$ by $E_{\delta}$, where $E_{\delta}$ is defined by \eqref{MET}. 	
\end{rmq}

\begin{rmq}
We do not manage to get a better error term in the range $q\ll X^{\frac{1}{2}}$ because the result in Proposition \ref{problemconv} involves the growth of the derivatives of $\omega$ \textit{via} the parameter $P=XH^{-1}$. As $q$ gets smaller, this quantity gets bigger and the 'trivial' bound gives in this case a better estimate.
\end{rmq}

\section{Variance of the divisor function in arithmetic progressions}\label{divisor}
The goal of this section is to prove Theorem \ref{thdivisor}.
\smallbreak
\noindent
The proof is very similar to the one of Theorem \ref{moduleqqc} but the details are more involved.
Recall the definition of $A(X,q;\tau)$ in \eqref{divisorAP} and \eqref{TPbxq}. By \cite{LauZhao}*{Equation~(4.12)} we can write
\begin{equation}\label{splitA(X,q,tau)}
	A(X,q;\tau)=A_w(q;\tau)+O\left(A_w(q;\tau)^{\frac{1}{2}}E_1^{\frac{1}{2}}+E_1\right),
\end{equation}
where
\begin{equation*}
	A_w(q;\tau)=\sum_{b(q)}\abs*{\sum_{n\equiv b[q]}\tau(n)w(n)-T_w(b,q)}^2,
\end{equation*}
\begin{equation*}
	T_w(b,q)=T_b(X,q)+O\left(q^{-1}\tau(q)^{\frac{1}{2}}\tau(b)^{\frac{1}{2}}H\log(X)\right)
\end{equation*}
and
\begin{equation*}
	E_1\ll \tau(q)\log(X)^3H^{2}q^{-1}. 
\end{equation*}
By \cite{LauZhao}*{Equation~(4.14)} we have
\begin{equation*}
	A_w(q;\tau)=\frac{1}{q}\sum_{r\mid q}\frac{1}{r^2}\sums{h(r)}\abs*{\sum_{n=1}^{\infty}a(n)e\left(-\frac{\overline{h}n}{r}\right)\omega_Y\left(\frac{\sqrt{n}}{r}\right)+\sum_{n=1}^{\infty}a(n)e\left(\frac{\overline{h}n}{r}\right)\omega_K\left(\frac{\sqrt{n}}{r}\right)}^2.
\end{equation*}
We expand the square to get
\begin{equation}\label{Aw=(1)+(2)+(3)}
	A_{w}(q;\tau)=(1)+(2)+E'(X,q,Y,K)
\end{equation}
where
\begin{equation*}
	(1)= \frac{1}{q}\sum_{r\mid q}\frac{1}{r^2}\sums{h(r)}\abs*{\sum_{n=1}^{\infty}\tau(n)e\left(-\frac{\overline{h}n}{r}\right)\omega_Y\left(\frac{\sqrt{n}}{r}\right)}^2,
\end{equation*}
\begin{equation*}
	(2)= \frac{1}{q}\sum_{r\mid q}\frac{1}{r^2}\sums{h(r)}\abs*{\sum_{n=1}^{\infty}\tau(n)e\left(\frac{\overline{h}n}{r}\right)\omega_K\left(\frac{\sqrt{n}}{r}\right)}^2,
\end{equation*}
and
\begin{equation*}
	E'(X,q,Y,K)=\frac{2}{q}\sum_{r\mid q}\frac{1}{r^2}\sums{h(r)}\Re\left(\sum_{n,m=1}^{\infty}\tau(n)\tau(m)e\left(\frac{\overline{h}(n+m)}{r}\right)\omega_Y\left(\frac{\sqrt{n}}{r}\right)\omega_K\left(\frac{\sqrt{m}}{r}\right)\right).
\end{equation*}
We apply the same argument of Proposition \ref{Awsplit} to the terms $(1)$ and $(2)$ to get for $B\in\{Y,K\}$
\begin{equation*}
	(1),(2)=\mathrm{MT}(X,q,B)+E(X,q,B)+O(X^{-2026}),
\end{equation*}  
where
\begin{equation*}
	\mathrm{MT}(X,q,B)=\frac{1}{q}\sum_{r\mid q}\frac{\phi(r)}{r^2}\sum_{n=1}^{\infty}\tau(n)^2\omega_B\left(\frac{\sqrt{n}}{r}\right)^2
\end{equation*}
and 
\begin{equation*}
	E(X,q,B)=\frac{1}{q}\sum_{dr\mid q}\frac{\mu(d)}{d^2r}\sum_{\substack{1\leq n,m\leq d^2r^2XH^{-2}X^{\eps}\\ n\equiv m[r]\\
	n\neq m\phantom{-}}} \tau(n)\tau(m)\omega_B\left(\frac{\sqrt{n}}{dr}\right)\omega_B\left(\frac{\sqrt{m}}{dr}\right).
\end{equation*}

\subsection{Main Term}
As in section \ref{Mainterm}, we will study the main term using inverse Mellin transform. We will prove the following.
\begin{proposition}\label{mainttermtau}
	Let $B\in\{Y,K\}$. Let $\delta\in[0;\frac{1}{12}]$ and assume \eqref{hypo}. We have
\begin{equation*}
	\mathrm{MT}(X,q,B)=\frac{X}{q}\sum_{r\mid q}\phi(r)\widetilde{Q_B}\left(\frac{r^2}{X}\right)+O\left(X^{\frac{3}{4}}H^{\frac{1}{4}}\log(X)^3+q^{-1}X^{\frac{3}{2}}g(q)\right)
\end{equation*}
where $\widetilde{Q_B}$ is a polynomial in $\log(\xi)$ with positive leading coefficient. 
\end{proposition}
\noindent As in section \ref{proofTh}, we first consider 
\begin{equation*}
	\sum_{n=1}^{\infty}\tau(n)^2\omega_B\left(\frac{\sqrt{n}}{r}\right)^2,
\end{equation*}
for $B\in\{Y,K\}$. 
\begin{proposition}\label{TPY,K}
	Let $B\in\{Y,K\}$. We have
\begin{equation*}
	\sum_{n=1}^{\infty}\tau(n)^2\omega_B\left(\frac{\sqrt{n}}{r}\right)^2=r^2 X\widetilde{Q_B}\left(\frac{r^2}{X}\right)+O(r^2X^{\frac{3}{4}}H^{\frac{1}{4}}\log(Xr)^3+rX^{\frac{3}{2}}), 
\end{equation*}
where $\widetilde{Q_B}(\xi)$ is a polynomial in $\log(\xi)$ of degree $3$ with positive leading coefficient.
\end{proposition}
\begin{preuve}
	We give the details for $B=Y$. The proof is similar for $B=K$. Using Proposition \ref{propomega} we can write
\begin{equation*}
	\sum_{n=1}^{\infty}\tau(n)^2\omega_Y\left(\frac{\sqrt{n}}{r}\right)^2=\frac{1}{(2i\pi)^2}\int_{(\sigma)}\int_{(\sigma')}G(s)r^{2s}G(s')r^{2s'}\psi(1-s)\psi(1-s')\frac{\zeta^4(s+s')}{\zeta(2(s+s'))}\ds\ds',
\end{equation*}
with $\sigma=\frac{1}{3}$ and $\sigma'=1$. Here we used the formula (see \cite{Titch}*{Equation~(1.2.10)})
\begin{equation*}
	\forall s\in\C,\Re(s)>1\quad \sum_{n=1}^{\infty}\frac{\tau(n)^2}{n^s}=\frac{\zeta^4(s)}{\zeta(2s)}.
\end{equation*}
As in Proposition \ref{Mainterm}, we apply the residue theorem to get
\begin{equation}
	\sum_{n=1}^{\infty}\tau(n)^2\omega_Y\left(\frac{\sqrt{n}}{r}\right)^2=I+E
\end{equation}
with 
\begin{equation} \label{TPY}
	I= \frac{1}{2i\pi}\int_{(\sigma)}G(s)r^{2s}\psi(1-s)\Res\left(G(s')r^{2s'}\psi(1-s') \frac{\zeta^4(s+s')}{\zeta(2(s+s'))},s'=1-s\right)\ds,
\end{equation}
\begin{equation}\label{errdinttau}
	E=\frac{1}{(2i\pi)^2}\int_{(\sigma)}\int_{(\sigma')}G(s)r^{2s}G(s')r^{2s'}\psi(1-s)\psi(1-s') \frac{\zeta^4(s+s')}{\zeta(2(s+s'))}\ds\ds', 
\end{equation}
and $0<\sigma'<1-\sigma$. We first compute the integral. $I$ 
We write
\begin{equation*}
	\frac{\zeta(s)^4}{\zeta(2s)}=\frac{a}{(s-1)^4}+\frac{b}{(s-1)^3}+\frac{c}{(s-1)^2}+\frac{d}{s-1}+O(1),
\end{equation*}
for some $a,b,c,d\in\R$. 
We now compute the residue in \eqref{TPY} to get
\begin{align*}
	&\Res\left(\frac{\zeta^4(s+s')}{\zeta(2(s+s'))}G(s')r^{2s'},s'=1-s\right)\\
	&=\frac{a}{6} \left.\left(G(s')\psi(1-s')r^{2s'}\right)^{(3)}\right\rvert_{s'=1-s}+\frac{b}{2}\left.\left(G(s')\psi(1-s')r^{2s'}\right)^{(2)}\right\rvert_{s'=1-s}\\
	&\quad +c\left.\left(G(s')\psi(1-s')r^{2s'}\right)^{(1)}\right\rvert_{s'=1-s}+dG(1-s)\psi(s)r^{2-2s}.
\end{align*}   
Let
\begin{equation*}
	\widetilde{G}(s)=G(s)\psi(1-s).
\end{equation*}
If we compute all the derivatives, we need to handle integrals of the form
\begin{equation*}
	 r^2\log(r^2)^l\frac{1}{2i\pi}\int_{(\sigma)} \widetilde{G}^{(m)}(1-s)\widetilde{G}(s)\ds,
\end{equation*}
for some $(l,m)\in[\![0;3]\!]^2$. To compute these integrals, we will use Parseval's formula (Proposition \ref{persevaltp}) and start with the case $m=0$. Recall \eqref{propomegaK}
\begin{equation*}
	\omega_Y(\sqrt{\xi})=\frac{1}{2i\pi}\int_{(\sigma)}G(s)\psi(1-s)\xi^{-s}\ds=\frac{1}{2i\pi}\int_{(\sigma)}\widetilde{G}(s)\xi^{-s}\ds.
\end{equation*}
Hence $s\mapsto\widetilde{G}(s)$ is the Mellin transform of $\xi\mapsto\omega_Y(\sqrt{\xi})$. By Mellin inversion we have for $\Re(s)>0$
\begin{equation*}
	\widetilde{G}(s)=\int_0^{\infty}\omega_Y(\sqrt{\xi})\xi^{s-1}\d\xi. 
\end{equation*}
Hence, by Parseval's formula (Proposition \ref{Parseval})
\begin{equation*}
	\frac{1}{2i\pi}\int_{(\sigma)}\widetilde{G}(1-s)\widetilde{G}(s)\ds=\int_0^{\infty}\omega_Y(\sqrt{\xi})^2\d\xi. 
\end{equation*}
Now if $m\geq 1$, we just need to see that
\begin{equation*}
	\widetilde{G}^{(m)}(s)=\int_0^{\infty}\log(\xi)^{m}\omega_Y(\sqrt{\xi})\xi^{s-1}\d\xi, 
\end{equation*}
so $s\mapsto \widetilde{G}^{(m)}(s)$ is the Mellin transform of $\xi\mapsto\log(\xi)^m\omega_Y(\sqrt{\xi})$
so Parseval's formula yields
\begin{equation*}
	\frac{1}{2i\pi}\int_{(\sigma)}\widetilde{G}^{(m)}(1-s)\widetilde{G}(s)\ds=\int_0^{\infty}\log(\xi)^{m}\omega_Y(\sqrt{\xi})^2\d\xi.
\end{equation*}
Finally we have
\begin{equation*}
	\sum_{n=1}^{\infty}\tau(n)^2\omega_Y\left(\frac{\sqrt{n}}{r}\right)^2=r^2\int_0^{\infty}Q_Y(\xi r^2)\omega_Y(\sqrt{\xi})^2\d\xi+E,
\end{equation*}
where $Q_Y(x)$ is a polynomial in $\log(x)$ of degree $3$ with positive leading coefficient. We now use Lemma \ref{tptau} to write
\begin{equation*}
	\int_0^{\infty}Q_Y(\xi r^2)\omega_Y(\sqrt{\xi})^2\d\xi= X\widetilde{Q_Y}\left(\frac{r^2}{X}\right)+O(X^{\frac{3}{4}}H^{\frac{1}{4}}\log(Xr)^3),
\end{equation*}
where $\widetilde{Q_Y}(\xi)$ is a polynomial in $\log(\xi)$ of degree $3$ whose leading coefficient has the same sign as $Q$. 
\smallbreak
\noindent
We now turn to the error term $E$, recall
\begin{equation*}
	E=\frac{1}{(2i\pi)^2}\int_{(\sigma)}\int_{(\sigma')}G(s)r^{2s}G(s')r^{2s'}\psi(1-s)\psi(1-s')\frac{\zeta^4(s+s')}{\zeta(2(s+s'))}\ds\ds'.
\end{equation*}
We shift both contours to $\sigma=\sigma'=\frac{1}{4}$, Stirling's formula (Lemma \ref{Stirling}) gives
\begin{equation*}
	\abs{G(\sigma+it)\psi(1-\sigma-it)}\asymp \frac{X^{\frac{3}{4}}}{(1+\abs{t})^{\frac{3}{2}}}.
\end{equation*}
Using Lemma \ref{moment4zeta} we easily show 
\begin{equation*}
	\int_{0}^{T}\int_0^{T'}\abs*{\zeta^4\left(\frac{1}{2}+it+it'\right)}\dt\dt'\ll (TT')^{1+\eps},
\end{equation*}
and finally, by carrying out the same analysis as in Proposition \ref{suma(n)^2} we have
\begin{equation*}
	E\ll rX^{\frac{3}{2}},
\end{equation*}
which ends the proof. 
\end{preuve}
\begin{rmq}\label{divisorRH}
	Assume \eqref{hypo} for some $\delta\in[0;\frac{1}{12}]$. If we follow the same reasoning as in Remark \ref{TEdelta} we can replace $q^{-1}X^{\frac{3}{2}}g(q)$ by $E_{\delta}$, where $E_{\delta}$ is defined by \eqref{MET}. 
\end{rmq}
We can now prove Proposition \ref{mainttermtau}.\smallbreak
\noindent
\begin{preuve}
	Recall
\begin{equation*}
	\mathrm{MT}(X,q,B)=\frac{1}{q}\sum_{r\mid q}\frac{\phi(r)}{r^2}\sum_{n=1}^{\infty}\tau(n)^2\omega_B\left(\frac{\sqrt{n}}{r}\right)^2.
\end{equation*}
Inserting Proposition \ref{TPY,K} we get
\begin{align*}
	\mathrm{MT}(X,q,B)&=\frac{1}{q}\sum_{r\mid q}\frac{\phi(r)}{r^2}\left(r^2 X\widetilde{Q_B}\left(\frac{r^2}{X}\right)+O(r^2X^{\frac{3}{4}}H^{\frac{1}{4}}\log(Xr)^3+rX^{\frac{3}{2}})\right)\\
	&=\frac{X}{q}\sum_{r\mid q}\phi(r)\widetilde{Q_B}\left(\frac{r^2}{X}\right)+O\left(X^{\frac{3}{4}}H^{\frac{1}{4}}\log(X)^3+q^{-1}X^{\frac{3}{2}}g(q)\right),
\end{align*}
where $\widetilde{Q_B}(x)$ is a polynomial in $\log(x)$ of degree $3$ with positive leading coefficient.
\end{preuve}

\subsection{Error terms}\label{termeerreurtau}
We now turn to the error term $E(X,q,Y)$, $E(X,q,K)$ and $E'(X,q,Y,K)$. We start with the first two. Recall
\begin{equation*}
	E(X,q,Y)=\frac{1}{q}\sum_{dr\mid q}\frac{\mu(d)}{d^2r}\sum_{\substack{1\leq n,m\leq d^2r^2XH^{-2}X^{\eps}\\ n\equiv m[r]\\
	n\neq m\phantom{-}}} \tau(n)\tau(m)\omega_Y\left(\frac{\sqrt{n}}{dr}\right)\omega_Y\left(\frac{\sqrt{m}}{dr}\right).
\end{equation*}
Let us define for $(B,B')\in\{Y,K\}^2$, and $x,y,h>0$ 
\begin{equation*}
	g_{B,B'}(x,y;h)=\omega_B\left(\frac{\sqrt{x}}{h}\right)\omega_{B'}\left(\frac{\sqrt{y}}{h}\right).
\end{equation*}
We will prove the following.
\begin{proposition}\label{EY}
	We have
\begin{equation*}
	E(X,q,Y)=\frac{2}{q}\sum_{dr\mid q}\frac{\mu(d)}{d^2r}\sum_{l=1}^{\infty}\int_{0}^{\infty}g_{Y,Y}(x,x+lr;dr)\Lambda_{lr}(x,x+lr)\dx+O(q^{\frac{5}{6}}X^{\frac{4}{3}}H^{-\frac{5}{4}}X^{\eps}). 
\end{equation*}
\end{proposition}
\begin{preuve}
We proceed exactly in the same way as in section \ref{termeerreur} and write
\begin{equation*}
	E(X,q,Y)=\frac{1}{q}\sum_{dr\mid q}\frac{\mu(d)}{d^2r}\sumd{M,N}\mathcal{E}(M,N,dr,r;\tau)
\end{equation*}
with
\begin{equation*}
	\mathcal{E}(M,N,h,r;\tau)= \sum_{\substack{1\leq n,m\leq h^2XH^{-2}X^{\eps}\\ n\equiv m[r]\\
	n\neq m\phantom{-}}} \tau(n)\tau(m)\omega_Y\left(\frac{\sqrt{n}}{h}\right)\rho\left(\frac{n}{N}\right)\omega_Y\left(\frac{\sqrt{m}}{h}\right) \rho\left(\frac{m}{M}\right).
\end{equation*}
As in section \ref{termeerreur}, we bound trivially the error term when $M\leq \Omega$
\begin{equation*}
	\frac{1}{q}\sum_{dr\mid q}\frac{\mu(d)}{d^2r}\sumd{M\leq \Omega}\sumd{N}\mathcal{E}(M,H,h,r;\tau)\ll q^{\frac{1}{2}}X^{\frac{3}{4}}H^{-\frac{1}{2}}\Omega^{\frac{1}{4}}X^{\eps}. 
\end{equation*}
We write $n\equiv m[r]$ as $n-m=lr$ to get
\begin{equation*}
	\mathcal{E}(M,N,h,r;\tau)=\sum_{l\neq 0}\sum_{n-m=lr} \tau(n)\tau(m)\omega_Y\left(\frac{\sqrt{n}}{h}\right)\rho\left(\frac{n}{N}\right)\omega_Y\left(\frac{\sqrt{m}}{h}\right) \rho\left(\frac{m}{M}\right).
\end{equation*}
We now use Proposition \ref{problemconv} to get 
\begin{align*}
	\mathcal{E}(M,N,h,r;\tau)&=\sum_{l\neq 0}\int_0^{\infty}g_{Y,Y}(x,x-lr;h)\rho\left(\frac{x}{N}\right)\rho\left(\frac{x-lr}{M}\right)\Lambda_{lr}(x,x-lr)\dx\\
	&+O(\max(N,M)r^{-1}AP^{\frac{5}{4}}(M+N)^{\frac{1}{4}}(MN)^{\frac{1}{4}+\eps}),
\end{align*}
where $A=h^3X^{\frac{1}{2}}(MN)^{-\frac{3}{4}}$, $P=XH^{-1}$,
and
\begin{equation*}
	\Lambda_{lr}(x,y)=\sum_{k=1}^{\infty}\frac{1}{k^2}c_k(lr)(\log(x)-2\gamma-2\log(k))(\log(y)-2\gamma-2\log(k)). 
\end{equation*}
We have
\begin{equation*}
	\frac{1}{q}\sum_{dr\mid q}\frac{\mu(d)}{d^2r}\sumd{M\geq \Omega}\sumd{N\geq \Omega}O(\max(N,M)r^{-1}AP^{\frac{5}{4}}(M+N)^{\frac{1}{4}}(MN)^{\frac{1}{4}+\eps})\ll q^{\frac{3}{2}}X^{\frac{5}{2}}H^{-\frac{11}{4}}\Omega^{-\frac{1}{2}}X^{\eps}
\end{equation*}
and
\begin{align*}
	&\frac{1}{q}\sum_{dr\mid q}\frac{\mu(d)}{d^2r}\sumd{M\leq \Omega}\sumd{N}\sum_{l\neq 0}\int_0^{\infty} g_{Y,Y}(x,x-lr;dr)\rho\left(\frac{x}{N}\right)\rho\left(\frac{x-lr}{M}\right)\Lambda_{lr}(x,x-lr)\dx\\
	&\ll  \frac{1}{q}\sum_{dr\mid q}\frac{\mu(d)}{d^2r}\sumd{M\leq \Omega}\sumd{N}\max(N,M)r^{-1}(dr)^3X^{\frac{1}{2}}(NM)^{-\frac{3}{4}}\min(N,M)\\
	&\ll q^{\frac{1}{2}}X^{\frac{3}{4}}H^{-\frac{1}{2}}\Omega^{\frac{1}{4}}X^{\eps}.
\end{align*}
Here, we used
\begin{equation}\label{bornelambda}
	\Lambda_{lr}(x,y)\ll \log(X)^2\sum_{k=1}^{\infty}\frac{c_{k}(lr)}{k^2}\log(k)\ll \log(X)^2\sum_{k=1}^{\infty}\frac{(lr,k)\log(k)}{k^2}\ll \log(X)^2\tau(lr)\ll X^{\eps}
\end{equation}
and 
\begin{equation*}
	\forall x\geq 1,\quad \sum_{k\leq x}(lr,k)=\sum_{d\mid lr}d\sum_{\substack{k\leq x\\ \hspace{0.22cm} (lr,k)=d}}\hspace{-0.3cm}1\leq \sum_{d\mid lr}d\sum_{k\leq \frac{x}{d}}1\leq x\tau(lr),
\end{equation*}
where the implied constant does not depend on $lr$.\smallbreak
\noindent 
Finally we have,
\begin{align*}
	E(X,q,Y)&=\frac{1}{q}\sum_{dr\mid q}\frac{\mu(d)}{d^2r}\sumd{M,N}\sum_{l\neq 0} \int_{\max(0,lr)}^{\infty}g_{Y,Y}(x,x-lr;dr)\rho\left(\frac{x}{N}\right)\rho\left(\frac{x-lr}{M}\right)\Lambda_{lr}(x,x-lr)\dx\\
	&\quad\quad +O(q^{\frac{1}{2}}X^{\frac{3}{4}}H^{-\frac{1}{2}}\Omega^{\frac{1}{4}}X^{\eps}+ q^{\frac{3}{2}}X^{\frac{5}{2}}H^{-\frac{11}{4}}\Omega^{-\frac{1}{2}}X^{\eps}).
\end{align*}
Recall the dyadic partitions
\begin{equation*}
	\sumd{M}\rho\left(\frac{x-lr}{M}\right)=\sumd{N}\rho\left(\frac{x}{N}\right)=1,
\end{equation*}
hence
\begin{align*}
	E(X,q,Y)&=\frac{1}{q}\sum_{dr\mid q}\frac{\mu(d)}{d^2r}\sum_{l\neq 0}\int_{\max(0,lr)}^{\infty}g_{Y,Y}(x,x-lr;dr)\Lambda_{lr}(x,x-lr)\dx\\
	&+O(q^{\frac{1}{2}}X^{\frac{3}{4}}H^{-\frac{1}{2}}\Omega^{\frac{1}{4}}X^{\eps}+ q^{\frac{3}{2}}X^{\frac{5}{2}}H^{-\frac{11}{4}}\Omega^{-\frac{1}{2}}X^{\eps}). 
\end{align*}
As in section \ref{termeerreur}, we optimize both error terms by taking
\begin{equation*}
	\Omega=q^{\frac{4}{3}}X^{\frac{7}{3}}H^{-3},
\end{equation*}
and finally get
\begin{equation*}
	E(X,q,Y)=\frac{1}{q}\sum_{dr\mid q}\frac{\mu(d)}{d^2r}\sum_{l\neq 0}\int_{\max(0,lr)}^{\infty}g_{Y,Y}(x,x-lr;dr)\Lambda_{lr}(x,x-lr)\dx+O(q^{\frac{5}{6}}X^{\frac{4}{3}}H^{-\frac{5}{4}}X^{\eps}). 
\end{equation*}
If $l\geq 1$ we have
\begin{equation*}
	\int_{\max(0,lr)}^{\infty}g_{Y,Y}(x,x-lr;dr)\Lambda_{lr}(x,x-lr)\dx=\int_0^{\infty}g_{Y,Y}(x+lr,x;dr)\Lambda_{lr}(x+lr,x)\dx,
\end{equation*}
as $g_{Y,Y}(x,y;h)=g_{Y,Y}(y,x;h)$ and $\Lambda_{lr}(x,y)=\Lambda_{lr}(y,x)$.\smallbreak
\noindent 
If $l\leq -1$ we have
\begin{equation*}
	\int_{\max(0,lr)}^{\infty}g_{Y,Y}(x,x-lr;dr)\Lambda_{lr}(x,x-lr)\dx=\int_0^{\infty}g_{Y,Y}(x,x-lr;dr)\Lambda_{lr}(x,x-lr)\dx,
\end{equation*}
hence 
\begin{align*}
	\sum_{l\neq 0}\int_{\max(0,lr)}^{\infty}\int_0^{\infty}g_{Y,Y}(x,x-lr;dr)\Lambda_{lr}(x,x-lr)\dx &=\sum_{l=1}^{\infty}g_{Y,Y}(x,x+lr;dr)\Lambda_{lr}(x,x+lr)\dx\\
	&+ \sum_{l=-\infty}^{-1}\int_0^{\infty}g_{Y,Y}(x,x-lr;dr)\Lambda_{lr}(x,x-lr)\dx\\
	&=2\sum_{l=1}^{\infty}\int_0^{\infty}g_{Y,Y}(x,x+lr;dr)\Lambda_{lr}(x,x+lr)\dx,
\end{align*}
which ends the proof. 
\end{preuve}

We do the same with $E(X,q,K)$ and end up with the following statement.
\begin{proposition}\label{EK}
	We have
\begin{equation*}
	E(X,q,K)=\frac{2}{q}\sum_{dr\mid q}\frac{\mu(d)}{d^2r}\sum_{l=1}^{\infty}\int_{0}^{\infty}g_{K,K}(x,x+lr;dr)\Lambda_{lr}(x,x+lr)\dx+O(q^{\frac{5}{6}}X^{\frac{4}{3}}H^{-\frac{5}{4}}X^{\eps}).
\end{equation*}
\end{proposition}
We now turn to $E'(X,q,Y,K)$. Recall 
\begin{equation*}
	E'(X,q,Y,K) =\Re\left(\frac{2}{q}\sum_{r\mid q}\frac{1}{q}\sums{h(r)}\sum_{n,m=1}^{\infty}\tau(n)\tau(m)e\left(\frac{\overline{h}(n+m)}{r}\right)\omega_Y\left(\frac{\sqrt{n}}{r}\right) \omega_K\left(\frac{\sqrt{n}}{r}\right)\right).
\end{equation*}
We will prove the following.
\begin{proposition}\label{EYK}
	We have
\begin{equation}
	E'(X,q,Y,K) = \frac{2}{q}\sum_{dr\mid q}\frac{\mu(d)}{d^2r}\sum_{l=1}^{\infty}\int_0^{lr}g_{Y,K}(x,lr-x;dr)\Lambda_{lr}(x,lr-x)\dx+O\left(q^{\frac{5}{6}}X^{\frac{4}{3}}H^{-\frac{5}{4}}X^{\eps}\right).\label{first}
\end{equation}
\end{proposition}
\begin{preuve}
	Writing
\begin{equation*}
	\sum_{r\mid q}\frac{1}{r^2}\sums{h(r)}=\sum_{dr\mid q}\frac{\mu(d)}{d^2r^2}\sum_{h(r)}
\end{equation*}
and executing the $h$ sum we have
\begin{equation*}
	E'(X,q,Y,K) =\Re\left(\frac{2}{q}\sum_{dr\mid q}\frac{\mu(d)}{d^2r} \sum_{\substack{1\leq n,m\leq d^2r^2XH^{-2}X^{\eps}\\ n+m\equiv 0[r]\phantom{-}}} \hspace{-0.9cm}\tau(n)\tau(m)\omega_Y\left(\frac{\sqrt{n}}{dr}\right)\omega_K\left(\frac{\sqrt{m}}{dr}\right)\right).
\end{equation*}
We use a dyadic partition in both variables $n$ and $m$ and we have 
\begin{equation*}
	E'(X,q,Y,K)=\Re\left(\frac{2}{q}\sum_{dr\mid q}\frac{\mu(d)}{d^2r}\sumd{M,N}\sum_{\substack{n,m\leq (dr)^2XH^{-2}X^{\eps}\\ n+m\equiv 0[r]\phantom{-}}}\hspace{-0.9cm}\tau(n)\tau(m)\omega_Y\left(\frac{\sqrt{m}}{dr}\right) \rho\left(\frac{m}{M}\right)\omega_K\left(\frac{\sqrt{n}}{dr}\right) \rho\left(\frac{n}{N}\right)\right).
\end{equation*}
By the same arguments as in the proof of Proposition \ref{EY} we prove \eqref{first}.
\end{preuve}

\subsection{Fake main terms}
The goal of this section is to bound the three ``fake main terms'' arising from Section \ref{termeerreurtau}. Precisely, we will bound
\begin{equation}\label{FMTB}
	\frac{2}{q}\sum_{dr\mid q}\frac{\mu(d)}{d^2r}\sum_{l=1}^{\infty}\int_{0}^{\infty}g_{B,B}(x,x+lr;dr)\Lambda_{lq}(x,x+lr)\dx,
\end{equation}
for $B\in\{Y,K\}$ and 
\begin{equation}\label{FMTYK}
	\frac{2}{q}\sum_{dr\mid q}\frac{\mu(d)}{d^2r}\sum_{l=1}^{\infty}\int_0^{lr}g_{Y,K}(x,lr-x;dr)\Lambda_{lr}(x,lr-x)\dx.
\end{equation}
We first bound \eqref{FMTB} in the next proposition. 
\begin{proposition}\label{boundFMTB}
	We have for $B\in\{Y,K\}$,
\begin{equation*}
	E:=\frac{1}{q}\sum_{dr\mid q}\frac{\mu(d)}{d^2r}\sum_{l=1}^{\infty}\int_{0}^{\infty}g_{Y,Y}(x,x+lr;dr)\Lambda_{lq}(x,x+lr)\dx\ll q\tau(q)X^{\frac{1}{2}}H^{-\frac{1}{2}}\log(X)^{4+\eps}. 
\end{equation*}
\end{proposition}
\begin{preuve}
We give the details for $B=Y$. The proof is similar for $B=K$. Let 
\begin{equation*}
	\widetilde{E}=\sum_{l=1}^{\infty}\int_{0}^{\infty}g_{Y,Y}(x,x+lr;dr)\Lambda_{lr}(x,x+lr)\dx
\end{equation*}
Recall that
\begin{equation*}
	\Lambda_{lr}(x,y)=\sum_{k=1}^{\infty}\frac{c_k(lr)}{k^2}(\log(x)-2\gamma-2\log(k))(\log(y)-2\gamma-\log(k)).
\end{equation*}
We have by Lemma \ref{Diricraman}
\begin{equation*}
	\sum_{k=1}^{\infty}\frac{c_k(lr)}{k^2}=\frac{1}{\zeta(2)}\sigma_{-1}(lr),\quad -\sum_{k=1}^{\infty}\frac{c_k(lr)}{k^2}\log(k)=-\frac{\zeta'(2)}{\zeta(2)^2}\sigma_{-1}(lr)+\frac{1}{\zeta(2)}\sigma_{-1}^{(1)}(lr) 
\end{equation*}
and 
\begin{equation*}
	\sum_{k=1}^{\infty}\frac{c_k(lr)}{k^2}\log(k)^2=\frac{2\zeta'(2)^2-\zeta''(2)\zeta(2)}{\zeta(2)^3}\sigma_{-1}(lr)-2\frac{\zeta'(2)}{\zeta(2)^2} \sigma_{-1}^{(1)}(lr)+\frac{1}{\zeta(2)} \sigma_{-1}^{(2)}(lr).
\end{equation*}
Hence, we have 
\begin{align}
	\widetilde{E}&=A_1\sum_{l=1}^{\infty}\int_0^{\infty}g_{Y,Y}(x,x+lr;dr)(\log(x)-2\gamma)(\log(x+lr)-2\gamma)\sigma_{-1}(lr)\dx\label{TE1}\\
	&+A_2\sum_{l=1}^{\infty}\int_0^{\infty}g_{Y,Y}(x,x+lr;dr)(\log(x)+\log(x+lr)-4\gamma)\sigma_{-1}(lr)\dx\\
	&+A_3\sum_{l=1}^{\infty}\int_0^{\infty}g_{Y,Y}(x,x+lr;dr)(\log(x)+\log(x+lr)-4\gamma) \sigma_{-1}^{(1)}(lr)\dx\\
	&+\sum_{l=1}^{\infty}\int_0^{\infty}g_{Y,Y}(x,x+lr;dr)\left(A_4\sigma_{-1}(lr)+A_5 \sigma_{-1}^{(1)}(lr) +A_6 \sigma_{-1}^{(2)}(lr)\right)\dx,
\end{align}
for some $A_i\in\R$. We will first analyse in detail \eqref{TE1}. We first start with 
\begin{equation*}
	E_1:=\sum_{l=1}^{\infty}\int_0^{\infty}g_{Y,Y}(x,x+lr;dr)\sigma_{-1}(lr)\dx.
\end{equation*}
Recall the definition of $g_{Y,Y}$
\begin{equation*}
	g_{Y,Y}(x,x+lr;dr)=\omega_Y\left(\frac{\sqrt{x}}{dr}\right)\omega_Y\left(\frac{\sqrt{x+lr}}{dr}\right).
\end{equation*}
We now use \eqref{propomegaY} on $\omega_Y\left(\frac{\sqrt{x+lr}}{dr}\right)$ to write 
\begin{equation}\label{MellinGYY}
	g_{Y,Y}(x,x+lr;dr)=\frac{1}{2i\pi}\int_{(\sigma)}\omega_Y\left(\frac{\sqrt{x}}{dr}\right)G(s)\psi(1-s)(dr)^{2s}(x+lr)^{-s}\ds,
\end{equation}
where $\sigma>1$. 
We then write using \eqref{(A+B)lambda}
\begin{equation}
	(x+lr)^{-s}=\frac{1}{2i\pi\Gamma(s)}\int_{(c)}\Gamma(s+z)\Gamma(-z)x^{-s-z}(lr)^{z}\dz,
\end{equation}
where the line $[c-i\infty;c+i\infty]$ separates the poles of $\Gamma(s+z)$ from those of $\Gamma(-z)$ \textit{i.e.}
\begin{equation*}
	-\sigma<c<0. 
\end{equation*}
We choose $c\in ]-\sigma;-1[$. Combining \eqref{MellinGYY} and \eqref{(A+B)lambda} we have
\begin{equation*}
	E_1=\frac{1}{(2i\pi)^2}\int_0^{\infty}\omega_{Y}\left(\frac{\sqrt{x}}{dr}\right)\int_{(\sigma)}\int_{(c)}G(s)\psi(1-s)(dr)^{2s}r^{z}x^{-s-z}\frac{\Gamma(s+z)\Gamma(-z)}{\Gamma(s)}\sum_{l=1}^{\infty}\frac{\sigma_{-1}(lr)}{l^{-z}}\ds\dz\dx.
\end{equation*}
We have by Lemma \ref{sumsigma_a}
\begin{equation*}
	\sum_{n=1}^{\infty}\frac{\sigma_{-1}(lr)}{l^{-z}}=
	\zeta(-z)\zeta(1-z)h(-1,r,-z).
\end{equation*}
Hence
\begin{align*}
	E_1&=\frac{1}{(2i\pi)^2}\int_0^{\infty}\omega_{Y}\left(\frac{\sqrt{x}}{dr}\right)\int_{(\sigma)}\frac{G(s)\psi(1-s)(dr)^{2s}x^{-s}}{\Gamma(s)}\\
	&\underbrace{\int_{(c)}\Gamma(s+z)\Gamma(-z)r^zx^{-z}\zeta(-z)\zeta(1-z)h(-1,r,-z)\dz}_{I}\ds\dx. 
\end{align*}
We consider the $c$-integral, we shift the line to $c=-1+\delta$ for some $\delta>0$ and pick up the residue of $\zeta(-z)$ at $z=-1$. By the residue theorem we have
\begin{equation*}
	I=2i\pi\Gamma(s-1)r^{-1}x\zeta(2)h(-1,r,1)-\int_{(-1+\delta)}\Gamma(s+z)\Gamma(-z)r^zx^{-z}\zeta(-z)\zeta(1-z)h(-1,r,-z)\dz. 
\end{equation*}
Inserting this into $E_1$, we see that $E_1=E_{1,1}+E_{1,2}$ with 
\begin{align*}
	E_{1,1}&=\frac{\zeta(2)h(-1,r,1)r^{-1}}{2i\pi}\int_0^{\infty}\omega_{Y}\left(\frac{\sqrt{x}}{dr}\right)\int_{(\sigma)}\frac{G(s)\psi(1-s)(dr)^{2s}x^{1-s}}{\Gamma(s)} \Gamma(s-1)\ds\dx\\
	&=\frac{\zeta(2)h(-1,r,1)r^{-1}}{2i\pi} \int_0^{\infty}\omega_{Y}\left(\frac{\sqrt{x}}{dr}\right)\int_{(\sigma)}\frac{G(s)\psi(1-s)(dr)^{2s}x^{1-s}}{s-1}\ds\dx
\end{align*}
and 
\begin{align*}
	E_{1,2}&=\frac{1}{(2i\pi)^2}\int_0^{\infty}\omega_{Y}\left(\frac{\sqrt{x}}{dr}\right)\int_{(\sigma)}\frac{G(s)\psi(1-s)(dr)^{2s}x^{-s}}{\Gamma(s)}\\
	&\int_{(-1+\delta)}\Gamma(s+z)\Gamma(-z)r^zx^{-z}\zeta(-z)\zeta(1-z)h(-1,r,-z)\dz\ds\dx. 
\end{align*}
We shift both contours in $E_{1,1}$ and $E_{1,2}$ to $\sigma=1-2\delta$ and pick up no residue.
Let $s=1-2\delta+it$ and $z=-1+\delta+it'$ then by Lemmas \ref{Stirling} and \ref{majpsi}
\begin{equation*}
	\abs{G(s)}\asymp (1+\abs{t})^{1-4\delta},
\end{equation*} 
\begin{equation*}
	\psi(1-s)\ll \frac{X^{2\delta}}{\sqrt{4\delta^2+t^2}},
\end{equation*}
where both implied constants do not depend on $\delta$. 
Finally,
\begin{align*}
	E_{1,1}&\ll h(r)r^{-1}X^{2\delta}(dr)^{2-4\delta}\\
	&\quad \times \left(\int_0^{(dr)^2XH^{-2}}\frac{(dr)^{\frac{3}{2}}X^{\frac{1}{4}}\dx}{x^{\frac{3}{4}+2\delta}}+\int_{(dr)^2XH^{-2}}^{\infty}\frac{(dr)^{\frac{5}{2}}X^{\frac{3}{4}}H^{-1}\dx}{x^{\frac{5}{4}-2\delta}}\right)\int_{\abs{t}}\frac{(1+\abs{t})^{1-4\delta}}{4\delta^2+t^2}\dt\\
	&\ll h(r)r^{-1}(dr)^4X^{\frac{1}{2}}H^{-\frac{1}{2}}\left(\frac{X}{H}\right)^{4\delta}\int_{\abs{t}}\frac{(1+\abs{t})^{1-4\delta}}{4\delta^2+t^2}\dt,
	%&\ll qX^{\frac{1}{2}}H^{-\frac{1}{2}}\left(\frac{X}{H}\right)^{4\delta}\int_{\abs{t}}\frac{(1+\abs{t})^{1-4\delta}}{4\delta^2+t^2}\dt.
\end{align*}
with 
\begin{equation*}
	h(r)=\prod_{p\mid r}\left(1+\frac{1}{p-1}\right).
\end{equation*}
Here we used Lemma \ref{majomega} 1. with $j=1$ for $x\leq (dr)^2XH^{-2}$ and $j=2$ for $(dr)^2XH^{-2}\leq x$.\smallbreak
\noindent
We choose 
\begin{equation*}
	\delta=\frac{1}{4\log\left(\frac{X}{H}\right)}
\end{equation*} 
hence $\left(\frac{X}{H}\right)^{4\delta}\ll 1$ and it remains to estimate the $t-$integral. We have 
\begin{equation*}
	\int_{\abs{t}}\frac{(1+\abs{t})^{1-4\delta}}{(4\delta^2+t^2)^{\frac{1}{2}}}\dt\ll \frac{1}{\delta^2}\int_{\abs{t}\ll \delta}(1+\abs{t})^{1-4\delta}\dt+\int_{\abs{t}\gg \delta}\frac{(1+\abs{t})^{1-4\delta}}{t^2}\dt\ll \log(X),
\end{equation*}
and finally 
\begin{equation*}
	E_{1,1}\ll h(r)r^{-1}(dr)^4X^{\frac{1}{2}}H^{-\frac{1}{2}}\log(X).
\end{equation*}
We do exactly the same analysis with $E_{1,2}$ and end up with 
\begin{align*}
	E_{1,2}&\ll h(r)r^{-1}(dr)^4X^{\frac{1}{2}}H^{-\frac{1}{2}}\left(\frac{X}{H}\right)^{3\delta}\\
	&\quad \times \iint (1+\abs{t})^{1-4\delta}\frac{1}{(4\delta^2+t^2)^{\frac{1}{2}}}\frac{\abs{\Gamma(-\delta+it+it')}\abs{\Gamma(1-\delta-it')}}{\abs{\Gamma(1-2\delta+it)}}\abs{t'}^{\delta}\dt\dt',
\end{align*}
where we used
\begin{equation*}
	\zeta(1-z)\ll 1,\quad \zeta(-z)\ll \abs{t'}^{\delta},
\end{equation*}
and by \eqref{h(a,r,s)}
\begin{equation*}
	h(-1,r,-z)\ll h(r).
\end{equation*}
We will show the double integral, $I'$ (say), is $\ll \log(X)^{2+\eps}$ for any $\eps>0$. By Lemma \ref{Stirling} we have
\begin{equation*}
	I'\ll \iint \frac{(1+\abs{t})^{\frac{1}{2}-2\delta}\abs{t'}^{\frac{1}{2}}}{(4\delta^2+t^2)^{\frac{1}{2}}(1+\abs{t+t'})^{\frac{1}{2}+\delta}}e^{-\frac{\pi}{2}\abs{t+t'}}e^{-\frac{\pi}{2}\abs{t'}}e^{\frac{\pi}{2}\abs{t}}\dt\dt',
\end{equation*}
where the implied constant does not depend on $\delta$. 
We consider the integral over $t'$, which is, uniformly in $\delta$,  
\begin{equation*}
	\ll \int \abs{t'}^{\frac{1}{2}}e^{-\frac{\pi}{2}\abs{t'}}\dt',
\end{equation*} 
hence we can cut the integral at $\abs{t'}\ll \log(X)^{1+\eps}$ at a negligible cost. Finally we have 
\begin{align*}
	I'&\ll \frac{1}{\delta}\iint_{\abs{t}\ll \delta}\dt\dt'+ \iint_{\delta\ll\abs{t}\ll \log(X)^{1+\eps}}\frac{\abs{t}^{\frac{1}{2}}}{(1+\abs{t})^{\frac{1}{2}+2\delta}}\dt\dt'+ \iint_{\log(X)^{1+\eps}\ll\abs{t}}\frac{\abs{t'}^{\frac{1}{2}}e^{-\frac{\pi}{2}\abs{t'}}}{(1+\abs{t})^{1+2\delta}}\dt\dt'\\
	&\ll \log(X)^{1+\eps}+\log(X)^{2+\eps}+\log(\log(X))\\
	&\ll \log(X)^{2+\eps}.
\end{align*}
Finally we have for any $\eps>0$
\begin{equation*}
	E_1\ll h(r)r^{-1}(dr)^4X^{\frac{1}{2}}H^{-\frac{1}{2}}\log(X)^{2+\eps}
\end{equation*}
and 
\begin{align*}
	\frac{1}{q}\sum_{dr\mid q}\frac{1}{d^2r}E_1&\ll \frac{1}{q}\sum_{dr\mid q}h(r)d(dr)^2X^{\frac{1}{2}}H^{-\frac{1}{2}}\log(X)^{2+\eps}\\
	&\ll qh(q)\tau(q)X^{\frac{1}{2}}H^{-\frac{1}{2}}\log(X)^{2+\eps}\\
	&\ll q\tau(q)X^{\frac{1}{2}}H^{-\frac{1}{2}}\log(X)^{2+\eps},
\end{align*}
where we used Lemma \ref{borneh(n))}.\smallbreak
\noindent 
If we were to estimate 
\begin{equation*}
	\sum_{l=1}^{\infty}\int_0^{\infty}g_{Y,Y}(x,x+lr;dr)\log(x+lr)\sigma_{-1}(lr)\dx,
\end{equation*}
we would carry out the same analysis as for $E_1$ but instead of using \eqref{(A+B)lambda} we would use \eqref{log(A+B)(A+B)lambda}
\begin{align*}
	\log(x+lr)(x+lr)^{-s}&=\frac{\Gamma'(s)}{2i\pi\Gamma(s)^2}\int_{(c)}\Gamma(s+z)\Gamma(-z)x^{-s-z}(lr)^{z}\dz\\
	&-\frac{1}{2i\pi\Gamma(s)}\int_{(c)}\Gamma'(s+z)\Gamma(-z)x^{-s-z}(lr)^{z}\dz\\
	&+\frac{1}{2i\pi\Gamma(s)}\int_{(c)}\Gamma(s+z)\Gamma(-z)\log(x)x^{-s-z}(lr)^{z}\dz.
\end{align*}
It does not significantly change the details involved because 
\begin{equation*}
	\abs{\Gamma'(\sigma+it)}\asymp \abs{\log(t)}\abs{\Gamma(\sigma+it)},
\end{equation*} 
by Lemma \ref{Stirling}. Finally, if we replace $\sigma_{-1}(lr)$ by $\sigma_{-1}^{(1)}(lr)$ or $\sigma_{-1}^{(2)}(lr)$, here again, it does not significantly change the details involved since
\begin{equation*}
	\sum_{l=1}^{\infty}\frac{\sigma_{-1}^{(1)}(lr)}{l^{-z}}=-\zeta(-z)\zeta'(1-z)h(-1,r,-z)+\zeta(-z)\zeta(1-z)\left.\frac{\partial}{\partial a}h(a,r,-z)\right\rvert_{a=-1}
\end{equation*}
and 
\begin{align*}
	\sum_{l=1}^{\infty} \frac{\sigma_{-1}^{(2)}(lr)}{l^{-z}}&=\zeta(-z)\zeta''(1-z)h(-1,r,-z)-2\zeta(-z)\zeta'(1-z) \left.\frac{\partial}{\partial a}h(a,r,-z)\right\rvert_{a=-1}\\
	&+\zeta(-z)\zeta(1-z)\left.\frac{\partial^2}{\partial a^2}h(a,r,-z)\right\rvert_{a=-1}
\end{align*}
The differences are either $h(-1,r,z)$ being replaced by one of its derivatives, or the term $\zeta(1 - z)$ being replaced by $\zeta'(1 - z)$ or $\zeta''(1 - z)$. The first difference is negligible thanks to \eqref{h(a,r,s)} while the second one is negligible because we always have $\Re(1-z)>1$.  
\end{preuve}
We now bound \eqref{FMTYK} in the next proposition. 
\begin{proposition}
We have
\begin{align*}
	E&:=\frac{2}{q}\sum_{dr\mid q}\frac{\mu(d)}{d^2r}\sum_{l=1}^{\infty}\int_0^{lr}g_{Y,K}(x,lr-x;dr)\Lambda_{lr}(x,lr-x)\dx\\
	&\ll qX^{\frac{1}{2}}H^{-\frac{1}{2}}\tau(q)\log(X)^{3+\eps}+qX^{\frac{3}{4}}H^{-\frac{3}{4}}\tau(q)\log(X)^{\frac{5}{2}+\eps}.
\end{align*}
\end{proposition}
\begin{preuve}
Similarly to Proposition \ref{boundFMTB}, we need to bound 
\begin{equation*}
	\widetilde{E}=\frac{2}{q}\sum_{dr\mid q}\frac{\mu(d)}{d^2r}\sum_{l=1}^{\infty}\int_0^{lr}\omega_Y\left(\frac{\sqrt{x}}{dr}\right)\omega_K\left(\frac{\sqrt{lr-x}}{dr}\right)\sigma_{-1}(lr)\dx.
\end{equation*}
We use $\sigma_{-1}(lr)\leq \sigma_{-1}(l)\sigma_{-1}(r)$ and the change of variable $x\longleftrightarrow lr-x$ to get 
\begin{equation*}
	\widetilde{E}\ll q^{-1}\sum_{dr\mid q}\frac{\sigma_{-1}(r)}{d^2r}\sum_{l=1}^{\infty}\sigma_{-1}(l)\int_0^{lr}\omega_Y\left(\frac{\sqrt{lr-x}}{dr}\right) \omega_K\left(\frac{\sqrt{x}}{dr}\right).
\end{equation*}
The $l$-sum can be cut at $l\leq (dr)^2r^{-1}XH^{-2}X^{\eps}$ at a negligible cost. We will distinguish three cases, namely  
\begin{align*}
	1\leq &\, l\leq L:=(dr)^2r^{-1}\log(X)^{2+\eps}H^{-1},\\
	L\leq &\, l\leq (dr)^2r^{-1}XH^{-2},\\
	(dr)^2lr^{-1}XH^{-2}\leq &\, l\leq (dr)^2lr^{-1}XH^{-2}X^{\eps} 
\end{align*}
and denote by $\widetilde{E}_1, \widetilde{E}_2$ and $\widetilde{E}_3$ the corresponding value of $\widetilde{E}$.\smallbreak
\noindent   
We start with the case 	$l\leq L$. We use 1. of Lemma \ref{majomega} with $j=1$ on both $\omega_Y$ and $\omega_K$ and get 
\begin{equation*}
	\int_0^{lr}\omega_Y\left(\frac{\sqrt{lr-x}}{dr}\right) \omega_K\left(\frac{\sqrt{x}}{dr}\right)\dx\ll (dr)^3X^{\frac{1}{2}}\int_0^{lr}x^{-\frac{3}{4}}(lr-x)^{-\frac{3}{4}}\dx\ll (dr)^3X^{\frac{1}{2}}(lr)^{-\frac{1}{2}}. 
\end{equation*}
Finally,
\begin{equation*}
	\widetilde{E}_1\ll q^{-1}X^{\frac{1}{2}}\sum_{dr\mid q}\frac{\sigma_{-1}(r)(dr)^3r^{-\frac{1}{2}}}{d^2r}\sum_{l\leq L}\frac{\sigma_{-1}(l)}{l^{\frac{1}{2}}}\ll qX^{\frac{1}{2}}H^{-\frac{1}{2}}\log(X)^{1+\eps}\sum_{r\mid q}\frac{\sigma(r)}{r},
\end{equation*}
where we used
\begin{equation*}
	\sum_{n\leq x}\frac{\sigma_1(n)}{n^{\frac{1}{2}}}=\sum_{n\leq x}\frac{\sigma(n)}{n^{\frac{3}{2}}}\ll x^{\frac{1}{2}},
\end{equation*}
see \cite{Tenenbaum}*{Theorem~3.3}.\smallbreak
\noindent
We now turn to the second case, by 2. of Lemma \ref{majomega} we can can cut the integral $\int_0^{lr}\dx$ at $x=L$ at a negligible cost. We then use Lemma \ref{majomega}~1. with $j=1$ on both $\omega_Y$ and $\omega_K$ and get 
\begin{equation*}
	\int_0^L\omega_Y\left(\frac{\sqrt{x}}{dr}\right)\omega_{K}\left(\frac{\sqrt{lr-x}}{dr}\right)\dx\ll (dr)^3X^{\frac{1}{2}}\int_0^Lx^{-\frac{3}{4}}(lr-x)^{-\frac{3}{4}}\dx\ll (dr)^3X^{\frac{1}{2}}(lr)^{-\frac{3}{4}}L^{\frac{1}{4}}. 
\end{equation*} 
Finally,
\begin{equation*}
	\widetilde{E}_2\ll q^{-1}X^{\frac{1}{2}}\sum_{dr\mid q}\frac{\sigma_{-1}(r)(dr)^{\frac{7}{2}}H^{-\frac{1}{4}}\log(X)^{\frac{1}{2}+\eps}}{d^2r^2}\hspace{-0.6cm}\sum_{L\leq l\leq (dr)^2r^{-1}XH^{-2}}\frac{\sigma_{-1}(l)}{l^{\frac{3}{4}}}\ll qX^{\frac{3}{4}}H^{-\frac{3}{4}}\log(X)^{\frac{1}{2}+\eps}\sum_{r\mid q}\frac{\sigma(r)}{r^{\frac{5}{4}}}.
\end{equation*}
The last case is very similar to the second case. The only difference is that we use 1. of Lemma \ref{majomega} with $j=2$ on $\omega_Y$ and we get the same error term.  
\end{preuve}

\subsection{Final computations}
The goal of this section is to prove Theorem \ref{thdivisor}.\smallbreak
\noindent
\begin{preuve}
We use \eqref{Aw=(1)+(2)+(3)}, Propositions \ref{TPY,K},  \ref{EY}, \ref{EK}, \ref{EYK}, \ref{FMTB} and \ref{FMTYK} to write 
\begin{align*}
	A_w(q;\tau)&=\frac{X}{q}\sum_{r\mid q}\phi(r)\widetilde{Q}\left(\frac{r^2}{X}\right)\\	&+O\left(qX^{\frac{1}{2}}H^{-\frac{1}{2}}\tau(q)\log(X)^{4+\eps}+X^{\frac{3}{4}}H^{\frac{1}{4}}\log(X)^3+ q^{\frac{1}{4}}X^{\frac{3}{4}}\tau(q)\log(X)^{\frac{5}{2}+\eps}\right)\\
	&+O\left(q^{-1}X^{\frac{3}{2}}g(q)+q^{\frac{5}{6}}X^{\frac{4}{3}}H^{-\frac{5}{4}}X^{\eps}\right),
\end{align*}
where $\widetilde{Q}=\widetilde{Q_Y}+\widetilde{Q_K}$ is a polynomial of degree $3$ with positive leading coefficient.\smallbreak
\noindent 
Inserting this expression into \eqref{splitA(X,q,tau)}, the error term will be
\begin{align*}
	&\ll X^{\frac{1}{2}}Hq^{-\frac{1}{2}}\log(X)^{\frac{3}{2}}\tau(q)^{\frac{1}{2}}+H^2q^{-1}\log(X)^3\tau(q)+qX^{\frac{1}{2}}H^{-\frac{1}{2}}\tau(q)\log(X)^{4+\eps}+X^{\frac{3}{4}}H^{\frac{1}{4}}\log(X)^3\\
	&\quad + q^{\frac{1}{4}}X^{\frac{3}{4}}\tau(q)\log(X)^{\frac{5}{2}+\eps}+ q^{-1}X^{\frac{3}{2}}g(q) + q^{\frac{5}{6}}X^{\frac{4}{3}}H^{-\frac{5}{4}}X^{\eps}.
\end{align*}
We choose
\begin{equation*}
	H=\frac{1}{3}\max(q^{\frac{5}{9}}X^{\frac{7}{18}},q).
\end{equation*}
For $q\leq X^{\frac{7}{8}}$,
the error term is
\begin{align*}
	&\ll q^{\frac{1}{18}}X^{\frac{8}{9}}\log(X)^{\frac{3}{2}} \tau(q)^{\frac{1}{2}} + q^{\frac{1}{9}}X^{\frac{7}{9}}\log(X)^3\tau(q)+q^{\frac{13}{18}}X^{\frac{11}{36}}\tau(q)\log(X)^{4+\eps}+q^{\frac{5}{36}}X^{\frac{61}{72}}\log(X)^3\\
	&\quad +q^{\frac{1}{4}}X^{\frac{3}{4}}\tau(q)\log(X)^{\frac{5}{2}+\eps} +q^{-1}X^{\frac{3}{2}}g(q)+q^{\frac{5}{36}}X^{\frac{61}{72}}X^{\eps}
\end{align*}
which simplifies into
\begin{equation*}
	\ll q^{\frac{5}{36}}X^{\frac{61}{72}}X^{\eps}+q^{-1}X^{\frac{3}{2}}g(q).
\end{equation*}
For $q\geq X^{\frac{7}{8}}$ the error term is
\begin{align*}
	&\ll q^{\frac{1}{2}}X^{\frac{1}{2}}\log(X)^{\frac{3}{2}}\tau(q)^{\frac{1}{2}}+q^{\frac{1}{2}}X^{\frac{1}{2}}\tau(q)\log(X)^{4+\eps}+q^{\frac{1}{4}}X^{\frac{3}{4}}\log(X)^3\\
	&\quad +q^{\frac{1}{4}}X^{\frac{3}{4}}\tau(q)\log(X)^{\frac{5}{2}+\eps} +q^{-1}X^{\frac{3}{2}}g(q)+q^{-\frac{5}{12}}X^{\frac{4}{3}}X^{\eps}
\end{align*} 
which simplifies into
\begin{equation*}
	\ll q^{\frac{1}{2}}X^{\frac{1}{2}}\tau(q)\log(X)^{4+\eps} +q^{\frac{1}{4}}X^{\frac{3}{4}}(\log(X)^3+\log(X)^{\frac{5}{2}+\eps}\tau(q)), 
\end{equation*}
which ends the proof. 
\end{preuve}

\end{document}